\documentclass[11pt]{article}

\usepackage{graphicx} 
\usepackage[T1]{fontenc}
\usepackage[utf8]{inputenc}
\usepackage[english]{babel}
\usepackage{newlfont}
\usepackage{graphicx} 
\usepackage{booktabs}

\usepackage{color}
\usepackage{amssymb}
\usepackage{amsmath}
\usepackage{amsthm}
\usepackage[utf8]{inputenc}
\usepackage{algorithm}
\usepackage{algpseudocode}

\usepackage{url}

\usepackage{caption}
\usepackage{subcaption}
\usepackage{verbatim}
\usepackage{multicol,multirow}

\usepackage{tikz}
\usepackage{tikz-3dplot}
\usepackage{pgfplots}
\pgfplotsset{compat=newest}

\newtheorem{theorem}{Theorem}[section]
\newtheorem{proposition}{Proposition}[section]
\newtheorem{lemma}{Lemma}[section]
\newtheorem{corollary}{Corollary}[section]
\newtheorem{definition}{Definition}[section]

\DeclareMathOperator{\Span}{span}
\DeclareMathOperator{\argmin}{argmin}
\DeclareMathOperator{\rank}{rank}
\DeclareMathOperator{\rankrel}{ReLU-rank}
\DeclareMathOperator{\ranksg}{sign-rank}
\DeclareMathOperator{\sgn}{sign}

\DeclareMathOperator{\err}{err}
\DeclareMathOperator{\diag}{diag}

\usepackage{listings}
\usepackage{xcolor}

\definecolor{codegreen}{rgb}{0,0.5,0}
\definecolor{codeblue}{rgb}{0,0,0.8}
\definecolor{codegray}{rgb}{0.5,0.5,0.5}
\definecolor{backcolor}{rgb}{0.96,0.96,0.96}

\lstdefinestyle{codestyle}{
    language=Matlab,
    backgroundcolor=\color{backcolor},
    basicstyle=\ttfamily\small,
    keywordstyle=\color{codeblue},
    commentstyle=\color{codegreen},
    stringstyle=\color{purple},
    numbers=left,
    numberstyle=\tiny\color{codegray},
    numbersep=8pt,
    frame=single,
    rulecolor=\color{black},
    breaklines=true,
    showstringspaces=false,
    tabsize=4,
    columns=fullflexible,
    keepspaces=true,
    captionpos=b,
    literate={~}{{$\sim$}}1,
    morekeywords={sdpvar,optimize,sdpsettings,value}
}
\definecolor{brightpink}{rgb}{1.0, 0.0, 0.5}

\definecolor{mygreen}{rgb}{0.0, 0.5, 0.4}
\newcommand{\gs}[1]{{\color{mygreen} \textbf{GS:} #1}}

\title{The ReLU Matrix Decomposition: NP-Hardness, Ill-Posedness, and ReLU-Rank\thanks{GS and NG acknowledge the support by the European Union (ERC consolidator, eLinoR, no 101085607). The work of GS\ was partially supported by INdAM-GNCS through Progetti di Ricerca.}}
\author{Giovanni Seraghiti\thanks{Corresponding author. University of Mons, Rue de Houdain 9, 7000 Mons, Belgium, and 
 Dipartimento di Ingegneria Industriale, Universit\`a degli Studi di Firenze, Viale Morgagni 40/44, 50134 Firenze, Italia. Member of the INdAM Research Group GNCS.}   
\and
Arnaud Vandaele\thanks{University of Mons, Rue de Houdain 9, 7000 Mons, Belgium.}
\and
Nicolas Gillis\footnotemark[3]
}

\date{}

\begin{document}

\maketitle

\begin{abstract}
The ReLU Matrix Decomposition (RMD) approximates a sparse, nonnegative $m $-by-$ n$ matrix $X$ as $X \approx \max(0, WH)$, where $W$ is an $m $-by-$ r$ matrix, $H$ is an $r $-by-$ n$ matrix, and $r$ is the rank of the decomposition. The most common formulation, referred to as Fro-RMD, minimizes the Frobenius norm of $X - \max(0, WH)$ over $W$ and $H$. We prove that Fro-RMD is NP-hard, even for binary matrices and $r = 1$; to the best of our knowledge, this is the first complexity result for RMD. We also show that Fro-RMD is ill-posed, in the sense that its infimum may not be attained. We then turn to exact RMD, that is, $X = \max(0, WH)$, and define the ReLU-rank of $X$ as the smallest $r$ for which such a decomposition exists. We interpret exact RMDs geometrically as configurations of points and hyperplanes in an $r$-dimensional space, associated with the rows of $W$ and the columns of $H$, respectively. Building on this interpretation, we derive upper bounds on the ReLU-rank of diagonal, block-diagonal, and block-sparse matrices and, under an additional structural assumption, improve the best known bound for block-diagonal matrices. Finally, we propose two formulations of exact RMD as feasibility problems: a nonconvex quadratically constrained program, which can be tackled by global solvers such as Gurobi, and a formulation based on the latent model of Saul, (``A nonlinear matrix decomposition for mining the zeros of sparse data'', SIMODS, 2022), which we tackle using alternating projections for which we establish local convergence together with asymptotic convergence rates. We compare these two approaches on some numerical examples. 
  
\end{abstract}



\section{Introduction}

Nonlinear Matrix Decomposition (NMD) consists of approximating a given matrix $X \in \mathbb{R}^{m \times n}$ as $X \approx f(WH)$, where $f$ is a nonlinear function applied entrywise, that is, $[f(\Theta)]_{ij}=f(\Theta_{ij})$, $W  \in \mathbb{R}^{m \times r}$ and $H \in \mathbb{R}^{r \times n}$ with $r\ll\min(m,n)$.
The choice of the function $f$ depends on the properties of the data and the task to be achieved. Prominent examples include 
\begin{itemize}
    \item $f(x)=\sgn(x)$: the sign decomposition for binary data~\cite{delsarte1989low}, see Section~\ref{sec:exact_NMD};
    
    \item $f(x)=\max(0,x)$: the ReLU Matrix Decomposition (RMD), which is effective in handling sparse, nonnegative data~\cite{saul2022nonlinear,saul2022geometrical,seraghiti2023accelerated,gillis2025extrapolated}; 
    
    \item $f(x)=1/(1+e^{-x})$: the sigmoid decomposition, which is particularly appropriate for bounded data~\cite{Nguyen, Thayse2026};
    
    \item $f(x) = x^p$: the component-wise exponent factorization~\cite{gillis2026complexity}, with the component-wise square as a special case~\cite{lefebvre2024component}, which is closely related to the square root rank.
    It is also related to the more general Hadamard decomposition, $X\approx W_1H_1 \odot W_2 H_2$~~\cite{ciaperoni2024hadamard,gillis2026manifold}. 
    
\end{itemize}
More examples can be found in~\cite{awari2025alternating}. 
In this paper, we focus on RMD, which finds application in a variety of contexts, including dimensionality reduction~\cite{saul2022geometrical}, matrix completion with entries missing not at random~\cite{liu2024symmetric,naik2022truncated,gillis2025extrapolated}, sensor localization~\cite{gillis2025extrapolated}, neural network theory~\cite{mazumdar2019learning,goel2017reliably}, and learning permutation matrices~\cite{droge2023kissing}. 

A standard approach to compute RMDs consists of minimizing an error measure between the original matrix $X$ and $\max(0,WH)$.
A typical choice of error measure is the Frobenius error, with the resulting optimization problem dubbed Frobenius RMD (Fro-RMD):
\begin{equation}
  \min_{W,H}  \| X-\max(0,WH) \|_F^2.
  \label{eq:ReLU_NMD}
\end{equation}
Even though Fro-RMD is an intuitive and natural formulation, to the best of our knowledge, its computational complexity and well-posedness have not been studied.
We fill this gap in this paper.
Moreover, Fro-RMD is nonsmooth and nonconvex, and hence difficult to solve; thus, the majority of available algorithms for computing RMDs address an alternative formulation, dubbed Latent-RMD, where a latent variable $Z$, with $\max(0,Z)=X$, is introduced, and the Frobenius error between $Z$ and $WH$ is minimized instead.
The first and most intuitive algorithm designed for Latent-RMD is the Naive-RMD algorithm by Saul~\cite{saul2022nonlinear}. 
It consists of an alternating optimization strategy that solves one subproblem for $Z$ and one for $WH$.
Following the seminal work of Saul~\cite{saul2022nonlinear}, several papers have proposed efficient algorithms for computing RMDs; see, e.g.,~\cite{seraghiti2023accelerated,awari2025alternating,awari2024coordinate,wang2024momentum,wang2025accelerated,saul2026subzero}.
In contrast, the interpretation of RMDs has received far less attention. A first contribution in this direction is~\cite{saul2022geometrical}, which establishes a connection between RMDs and norm-preserving and neighborhood-preserving embeddings.
In this work, we focus on the geometric interpretation of exact RMD and propose an intuitive representation as a set of points and hyperplanes in the $r$-dimensional subspace spanned by the decomposition. 

The general intuition behind RMD is to exploit the sparsity of the original data matrix $X$ by finding a low-rank matrix that coincides with $X$ on its positive entries and is nonpositive elsewhere.
In this paper, given a nonnegative, sparse matrix $X$, we address the problem of finding the smallest $r$ such that $X$ admits an exact RMD of rank $r$.
We define the $\rankrel$ of a matrix as follows. 
\begin{definition}
    Let $X\in \mathbb{R}_+^{m \times n}$.
    The ReLU-rank of $X$, denoted $\rankrel(X)$, is the smallest $r$ such that there exists $\Theta\in \mathbb{R}^{m \times n}$ with $\rank(\Theta)=r$ and $X=\max(0,\Theta)$.  
    \label{def:f_rank}
\end{definition}
To our knowledge, the only contribution addressing exact RMD is the very recent paper by Saul et al.~\cite{saul2026subzero}, where they analyze several classes of matrices, including signless Laplacian matrices of sparse graphs and block diagonal matrices, providing bounds on their $\rankrel$. 
We use the proposed geometric representation of exact RMD to improve the known bound on the $\rankrel$ of block diagonal matrices; and we derive explicit upper bounds on the $\rankrel$ of general block-structured, sparse matrices.

Finally, we propose two reformulations of exact RMD as feasibility problems.
The first is a Nonconvex Quadratically Constrained Program (QCP), dubbed ReLU-QCP, that can be tackled by standard solvers, such as Gurobi.
The second consists of finding a point in the intersection of the set of matrices with rank at most $r$ and the set of matrices having the same positive entries as $X$ while being nonpositive in the other entries.
We address the second feasibility problem by using Naive-RMD, which we show can be interpreted as an alternating projection algorithm.
By using the available theory on alternating projections, we establish local convergence of Naive-RMD under the assumption that an exact decomposition exists.
We illustrate these two approaches on some numerical examples.

\paragraph*{Outline of the paper}
In Section~\ref{sec:ill_posedness}, we prove that Fro-RMD is ill-posed (Proposition~\ref{prop:ill_posed}) and NP-hard even in the rank-$1$ case (Theorem~\ref{thm:red_CUT_to_R1ReLUP}).
In Section~\ref{sec:exact_NMD}, we introduce a geometric interpretation of exact RMD as a points-hyperplanes system, which is employed to provide upper bounds on the $\rankrel$ for special classes of matrices, including diagonal (Theorem~\ref{th:relu_id}), block diagonal (Theorem~\ref{th:block_id_relu_aff}), and general block-structured (Theorem~\ref{th:kron}) matrices.
In Section~\ref{sec:algo}, we present the two feasibility problem reformulations of exact RMD. We also compare Naive-RMD and the Gurobi solver, each addressing one of the two formulations, across different classes of matrices.

\section{Ill-posedness and complexity of Fro-RMD}
\label{sec:ill_posedness}
In this section, we analyze the ill-posedness and the complexity of~Fro-RMD in~\eqref{eq:ReLU_NMD}.
Our contribution, together with~\cite[Example 1]{gillis2025extrapolated}, shows that both formulations used to compute RMDs, namely Fro-RMD and Latent-RMD, are ill-posed.

\subsection{Ill-posedness}\label{subsec:ill-pos}
To prove the ill-posedness of Fro-RMD, we show that there exists a $3 \times 3$ matrix such that the infimum of Fro-RMD is not attained when $r=2$. 
\begin{proposition}[Ill-posedness of Fro-RMD]\label{prop:ill_posed}
Let
\begin{equation*}
X=\begin{pmatrix} 1 & 1 & 1\\ 1 & 1 & 2\\ 1 & 2 & 0 \end{pmatrix}
\qquad\text{and}\qquad
F(\Theta)=\lVert X-\max(0,\Theta)\rVert_F^2 .
\end{equation*}
Then $\rankrel(X)=3$ and $\inf_{\rank(\Theta)\leq 2}F(\Theta)=0$, that is, the infimum of Fro-RMD with $r=2$ is not attained.
\end{proposition}
\begin{proof}
    First, we observe that $X$ cannot admit an exact RMD of $\rank$ at most 2; hence, \mbox{$\rankrel(X)=3$}. This follows from the fact that any matrix $ \Bar{\Theta}$ such that $X=\max(0, \Bar{\Theta})$ is of the form
    \begin{equation*}
   \Bar{\Theta}= \begin{pmatrix}
        1 & 1 & 1 \\
        1 & 1 & 2 \\
        1 & 2 & \alpha
    \end{pmatrix}, \quad \text{ with } \alpha \leq 0,
\end{equation*}
and has rank 3.
Indeed, one can verify that, independently of $\alpha \in \mathbb{R}$, we have $\det( \Bar{\Theta})=-1$.
Therefore, there cannot exist an exact RMD of rank less than or equal to 2. 
Next, for $t \in (0,1)$, define the family of matrices 
\begin{equation*}
   \Theta_t= \begin{pmatrix}
        1 & 1 & 1 \\
        1 & 1-t & 2 \\
        1 & 2 & 1-\frac{1}{t}
    \end{pmatrix}.
\end{equation*}
Every matrix $\Theta_t$ satisfies $\rank(\Theta_t)=2$ since $\det(\Theta_t)=0$ for every $t \in ]0,1[$ and the upper left minor has non-zero determinant. Moreover, for $t \in ]0,1[$, $1-\frac{1}{t} < 0$, hence
\begin{equation*}
    \max(0,\Theta_t)=\begin{pmatrix}
        1 & 1 & 1 \\
        1 & 1-t & 2 \\
        1 & 2 & 0
    \end{pmatrix} \quad \text{and} \quad F(\Theta_t)=\left\lVert \begin{pmatrix}
        0 & 0 & 0 \\
        0 & t & 0 \\
        0 & 0 & 0
    \end{pmatrix} \right\rVert_F^2=t^2.
\end{equation*} 
Consequently,  $\inf_{\rank(\Theta) \leq 2 }F(\Theta)=0$.
If the infimum were attained, $X$ would admit an exact RMD of rank at most 2, which we showed to be impossible.
Hence Fro-RMD is ill-posed.
\end{proof}

As a consequence of Proposition~\ref{prop:ill_posed}, we observe that  the sublevel set of Fro-RMD, that is, \newline \mbox{$\{ (W,H) \ | \ \lVert X-\max(0,WH) \rVert_F^2 \leq C \}$}, where $C$ is a positive scalar, is not compact.
The lack of compactness implies that when designing algorithms, the generated sequence of iterates $\{W_k\}_{k \in \mathbb{N}}$ and $\{H_k\}_{k \in \mathbb{N}}$ might be unbounded; that is, $\lVert W_k \rVert \to \infty$ or $\lVert H_k \rVert \to \infty$, while producing a convergent function value.
In practice, this issue can be addressed by penalizing the norms of $W$ and $H$ or by bounding their entries~\cite{wang2024momentum,wang2025accelerated}.
It does not seem to matter much in applications, since unbounded sequences appear to arise only in degenerate cases such as Proposition~\ref{prop:ill_posed}. 

\subsection{Complexity}
In this section, we prove the NP-hardness of rank-1 Fro-RMD, dubbed R1-RMD, even when the input matrix is binary.
Our proof proceeds by a polynomial time reduction from the decision version of the Cut-Norm problem. 
We begin by formally stating the decision versions of both problems.

\vspace{0.2cm}

\noindent
\fbox{\begin{minipage}{0.99\textwidth}\textbf{\textsc{R1-RMD}.}

\noindent
\emph{Given:} a matrix $X\in\{0,1\}^{m\times n}$ and $T\in\mathbb{R}_+$.

\noindent
\emph{Question:} does there exist $(w,h)\in\mathbb{R}^m\times\mathbb{R}^n$ such that $\|X-\max(0,wh^T)\|_F^2 \leq T$?
\end{minipage}}\\

\vspace{0.2cm}

\noindent
\fbox{\begin{minipage}{0.99\textwidth}\textbf{\textsc{Cut-Norm}.}

\noindent
\emph{Given:} a matrix $M\in\{\pm 1\}^{s\times t}$, and $D\in\mathbb{N}$.

\noindent
\emph{Question:} does there exist $(u,v)\in\{\pm 1\}^s\times\{\pm 1\}^t$ such that $\langle M,uv^T\rangle \geq D$?
\end{minipage}}\\

Cut-Norm is NP-hard by a polynomial-time reduction from the MAX-CUT problem~\cite[Theorem~1]{gillis2018complexity}.
Let us now reduce Cut-Norm to R1-RMD in polynomial-time. 
Given an instance $(M,D)$ of Cut-Norm, we construct an instance $(X,T)$ of R1-RMD by taking 
\begin{equation*}
    N = 3\cdot12^2s^3t^3, \quad T=\frac{st-D}{2}+N(s+t), \quad  m=s+2N, \quad n=t+2N,
\end{equation*}
and $X\in\{0,1\}^{m\times n}$ is defined as follows:
    \[
X=
\begin{pmatrix}
B & \mathbf 1_{s\times N} & \mathbf 1_{s\times N}\\[1mm]
\mathbf 1_{N\times t} & \mathbf 1_{N\times N} & \mathbf 0_{N\times N}\\[1mm]
\mathbf 1_{N\times t} & \mathbf 0_{N\times N} & \mathbf 1_{N\times N}
\end{pmatrix},
\]
where $\mathbf 1_{m \times n}$ denotes the $m $-by-$ n$ matrix of all ones, and $B=\frac{M + \mathbf 1_{s\times t}}{2}\in\{0,1\}^{s\times t}.$

Since $N$ is polynomial in $s$ and $t$, the construction is polynomial.

\begin{theorem}\label{thm:red_CUT_to_R1ReLUP}
The instance $(X,T)$ is a yes-instance of R1-RMD if and only if $(M,D)$ is a yes-instance of CUT-NORM. 
Hence R1-RMD is NP-hard. 
\end{theorem}

\begin{proof}
Let
\[
I=\{1,\dots,s\},\qquad
R=\{s+1,\dots,s+N\},\qquad
R'=\{s+N+1,\dots,s+2N\}
\]
be the indices of three row blocks, and
\[
J=\{1,\dots,t\},\qquad
C=\{t+1,\dots,t+N\},\qquad
C'=\{t+N+1,\dots,t+2N\}
\]
be the indices of three column blocks. For a vector $g$ and a matrix $G$, we denote by $g(I)$ and $G(I,J)$ the subvector and submatrix indexed by  $I$ and $I \times J$, respectively.

\paragraph{The \textit{if} part.}

\noindent
Suppose first that $(M,D)$ is a yes-instance of CUT-NORM.
Then there exist $u\in\{\pm1\}^s$ and
$v\in\{\pm1\}^t$ such that $u^TMv\geq D$.
We define $w\in\mathbb{R}^{s+2N}$ and $h\in\mathbb{R}^{t+2N}$ as follows:
\[
w=
\begin{pmatrix}
u\\
\mathbf 1_N\\
-\mathbf 1_N
\end{pmatrix}
\in\mathbb{R}^{m},
\qquad
h=
\begin{pmatrix}
v\\
\mathbf 1_N\\
-\mathbf 1_N
\end{pmatrix}
\in\mathbb{R}^{n}.
\]
With this definition of $(w,h)$, we examine the nine different parts of the matrix $X$ to show that
\[
\|X-\max(0,wh^T)\|_F^2 \leq T.
\]
\textbf{Case 1.} For $I\times J$.
    For $(i,j)\in I\times J$, we have $w_i=u_i$ and $h_j=v_j$ and since these entries belong to $\{\pm 1\}$, we have that
    \[
    \max(0,w_i h_j)=\max(0,u_iv_j) = \frac{u_iv_j+1}{2}.
    \]
    The contribution of the block $X(I,J)$ to the objective is:
\begin{eqnarray*}
\|X(I,J)-\max(0,w(I)h(J)^T)\|_F^2 & =  & \sum_{i=1}^s\sum_{j=1}^t\left(B_{i,j}-\max(0,w_ih_j)\right)^2 \\
& = &\sum_{i=1}^s\sum_{j=1}^t\left(\frac{M_{i,j}+1}{2}-\frac{u_iv_j+1}{2}\right)^2\\
& = & \frac{1}{4}\sum_{i=1}^s\sum_{j=1}^t (M_{i,j}^2-2u_iv_jM_{i,j}+(u_iv_j)^2) \\
& = & \frac{1}{4}\left(2st-2\sum_{i=1}^s\sum_{j=1}^tu_iv_jM_{i,j}\right)= \frac{st-\langle M,uv^T\rangle}{2}.
\end{eqnarray*}
    \textbf{Case 2.} For $R\times C$ and $R'\times C'$, the corresponding blocks of $\max(0,wh^\top)$ are all ones, hence
    \[\|X(R,C)-\max(0,w(R)h(C)^T)\|_F^2=\|X(R',C')-\max(0,w(R')h(C')^T)\|_F^2=0.\]
    \textbf{Case 4.} For $R\times C'$ and $R'\times C$, the corresponding blocks of $\max(0,wh^\top)$ are all zeros, hence
    \[\|X(R,C')-\max(0,w(R)h(C')^T)\|_F^2=\|X(R',C)-\max(0,w(R')h(C)^T)\|_F^2=0.\]
    \textbf{Case 5.} For $I\times(C\cup C')$, each row $i\in I$ must fit two sets of $N$ entries equal to $1$, corresponding to $C$ and $C'$.
    Since only one of $\max(0,w_i)$ and $\max(0,-w_i)$ can be active (and equal to $1$), exactly one set is matched, and the other contributes $N$ to the error:
    \begin{equation*}
        \begin{aligned}
           \|X(I,C\cup C')-\max(0,w(I)h(C\cup C')^T)\|_F^2
& =   \sum_{i=1}^s\left(\sum_{j\in C}\left(1-\max(0,w_i)\right)^2+\sum_{j\in C'}\left(1-\max(0,-w_i)\right)^2\right)\\
& =   N\sum_{i=1}^s\left(\left(1-\max(0,w_i)\right)^2+\left(1-\max(0,-w_i)\right)^2\right)\\
& =   N\sum_{i=1}^s\left(2-2|w_i|+|w_i|^2\right) =   N\sum_{i=1}^s\left(2-2+1^2\right) = Ns.
        \end{aligned}
    \end{equation*}
    \textbf{Case 6.} For $(R\cup R')\times J$, the situation is symmetric to the previous case, and we have
    \[
\|X(R\cup R',J)-\max(0,w(R\cup R')h(J)^T)\|_F^2 = Nt.
\]

Finally,
\[
\|X-\max(0,wh^T)\|_F^2
=
\frac{st-\langle M,uv^T\rangle}{2}+N(s+t)\leq \frac{st-D}{2}+N(s+t)
= T.
\]
Hence $(X,T)$ is a yes-instance of R1-RMD.

\medskip
\noindent
\textbf{The \textit{only if} part.}
Let $(X,T)$ be a yes-instance of R1-RMD, that is, there exist
$w\in\mathbb{R}^{m}$ and $h\in\mathbb{R}^{n}$ such that
\[
\|X-\max(0,wh^\top)\|_F^2\leq T.
\]
We show that this implies that $(M,D)$ is a yes-instance of \textsc{Cut-Norm}.

We first exploit the repeated rows and columns of $X$.
For a fixed $h$, all rows $X(i,:)$ with $i\in R$ are the same, hence the corresponding one-variable subproblems in the variables $w_i$, $i\in R$, are identical.
Therefore, replacing all $w_i$, $i\in R$, by the one with the smallest row error does not increase the objective; the same holds for $R'$, $C$ and $C'$. Hence, without loss of generality, we may assume that $(w,h)$ has the form
\[
w=
\begin{pmatrix}
a\\
\alpha \mathbf 1_N\\
\beta \mathbf 1_N
\end{pmatrix},
\qquad
h=
\begin{pmatrix}
b\\
\gamma \mathbf 1_N\\
\delta \mathbf 1_N
\end{pmatrix},
\]
for some $a\in\mathbb{R}^s$, $b\in\mathbb{R}^t$, and
$\alpha,\beta,\gamma,\delta\in\mathbb{R}$.

We now decompose the objective according to the four natural parts of $X$:
\[
\|X-\max(0,wh^T)\|_F^2 = E_{IJ}+E_{IR}+E_{CJ}+E_{\mathrm BR},
\]
where $E_{IJ}$ is the contribution of the top-left block $I\times J$,
$E_{IR}$ the contribution of the top-right blocks $I\times (C\cup C')$,
$E_{CJ}$ the contribution of the bottom-left blocks $(R\cup R')\times J$,
and $E_{\mathrm BR}$ the contribution of the bottom-right blocks
$(R\cup R')\times (C\cup C')$.

We also define $u\in\{\pm1\}^s$ and $v\in\{\pm1\}^t$ using the signs of $a$ and $b$:
\begin{equation}\label{eq:def_sigma_tau}
u_i=
\begin{cases}
+1,& a_i\geq 0,\\
-1,& a_i<0,
\end{cases}
\qquad
v_j=
\begin{cases}
+1,& b_j\geq 0,\\
-1,& b_j<0.
\end{cases}
\end{equation}
It remains to prove that $\langle M,uv^T\rangle\geq D$.

\medskip
\noindent\textbf{Step 1: the bottom-right blocks fix the signs of $\alpha,\beta,\gamma,\delta$.}~Since
\[ T=\frac{st-D}{2}+N(s+t)\le N(s+t)+st, \]
we have
\[
\|X-\max(0,wh^\top)\|_F^2\le N(s+t)+st.
\]
We first analyze the four $N\times N$ blocks in the bottom-right corner. Their
contribution is
\begin{eqnarray}
E_{\mathrm BR} & = &  \|X(R\cup R',C\cup C')-\max(0,w(R\cup R')h(C\cup C')^\top)\|_F^2 \notag\\
&=&
N^2\Bigl(
(1-\max(0,\alpha\gamma))^2
+\max(0,\alpha\delta)^2
+\max(0,\beta\gamma)^2
+(1-\max(0,\beta\delta))^2
\Bigr).
\label{eq:EBR}
\end{eqnarray}
Since $E_{\mathrm BR}\leq \|X-\max(0,wh^\top)\|_F^2$, we obtain
\[
\frac{E_{\mathrm BR}}{N^2} \leq \frac{N(s+t)+st}{N^2}
= \frac{s+t}{N}+\frac{st}{N^2} < \frac{s+t}{N}+\frac{st}{N} \leq \frac{2st}{N}+\frac{st}{N} = \frac{3st}{N} = \frac{1}{12^2s^2t^2}.
\]
Since all four terms in \eqref{eq:EBR} are nonnegative, each of them is strictly smaller
than $1/(12^2s^2t^2)$. Therefore,
\begin{eqnarray}
&&
\max(0,\alpha\gamma)\in\left(1-\frac{1}{12st},\,1+\frac{1}{12st}\right),
\qquad
\max(0,\beta\delta)\in\left(1-\frac{1}{12st},\,1+\frac{1}{12st}\right),
\label{eq:boundparameters}
\\
&&
\max(0,\alpha\delta)<\frac{1}{12st},
\qquad
\max(0,\beta\gamma)<\frac{1}{12st}.
\notag
\end{eqnarray}
The relations \eqref{eq:boundparameters} imply that $\alpha$ and $\gamma$ have the same
sign, and that $\beta$ and $\delta$ have the same sign. Moreover, $\alpha$ and $\delta$
must have opposite signs. Indeed, if they had the same sign, then
\[
\alpha\beta\gamma\delta=(\alpha\gamma)(\beta\delta)=(\alpha\delta)(\beta\gamma)>0,
\]
while at the same time
\[
(\alpha\delta)(\beta\gamma)<\left(\frac{1}{12st}\right)^2
\qquad\text{and}\qquad
(\alpha\gamma)(\beta\delta)>\left(1-\frac{1}{12st}\right)^2,
\]
a contradiction.
Hence $\alpha$ and $\delta$ have opposite signs, and therefore $\beta$ and $\gamma$ also have opposite signs.

Without loss of generality, we assume in the following that
\begin{equation}\label{eq:anchor_signs}
\alpha>0,\qquad \gamma>0,\qquad \beta<0,\qquad \delta<0.
\end{equation}
With these signs, \eqref{eq:EBR} simplifies to
\[
E_{\mathrm BR} = N^2\Bigl(
(1-\alpha\gamma)^2+(1-\beta\delta)^2
\Bigr),
\]
and \eqref{eq:boundparameters} gives
\begin{equation}\label{eq:parameters_interval}
\alpha\gamma,\beta\delta\in\left(1-\frac{1}{12st},\,1+\frac{1}{12st}\right).
\end{equation}

\medskip
\noindent\textbf{Step 2: the side blocks constrain the entries of $a$ and $b$.}~We next analyze the two rectangular blocks $I\times(C\cup C')$. Since $\gamma>0$ and
$\delta<0$, each row $i\in I$ can match only one of the two groups of $N$ columns,
depending on the sign of $a_i$. Therefore,
\begin{eqnarray*}
E_{IR}
& = & \|X(I,C\cup C')-\max(0,w(I)h(C\cup C')^\top)\|_F^2\\
& = & \sum_{i\in I}N\Bigl((1-\max(0,a_i\gamma))^2+(1-\max(0,a_i\delta))^2\Bigr)\\
& = & \sum_{i\in I:a_i\geq 0}N\Bigl(1+(1-a_i\gamma)^2\Bigr)
+
\sum_{i\in I:a_i<0}N\Bigl(1+(1-a_i\delta)^2\Bigr)\\
& = & Ns+\sum_{i\in I:a_i\ge0}N(1-a_i\gamma)^2+\sum_{i\in I:a_i<0}N(1-a_i\delta)^2.
\end{eqnarray*}
Similarly, for the two rectangular blocks $(R\cup R')\times J$, we obtain
\begin{eqnarray*}
E_{CJ} & = & \|X(R\cup R',J)-\max(0,w(R\cup R')h(J)^\top)\|_F^2\\
& = & Nt+\sum_{j\in J:b_j\geq 0}N(1-b_j\alpha)^2+\sum_{j\in J:b_j<0}N(1-b_j\beta)^2.
\end{eqnarray*}

Since
\[ E_{IR}+E_{CJ} \leq \|X-\max(0,wh^\top)\|_F^2 \leq N(s+t)+st,
\]
we have that
\[ \sum_{i\in I:a_i\geq 0}N(1-a_i\gamma)^2
+\sum_{i\in I:a_i<0}N(1-a_i\delta)^2
+\sum_{j\in J:b_j\geq 0}N(1-b_j\alpha)^2
+\sum_{j\in J:b_j<0}N(1-b_j\beta)^2
\leq st. \]
All terms in the left-hand side are nonnegative, hence each of them is at most $st$.
For example, when $a_i\geq 0$, we have $N(1-a_i\gamma)^2 \leq st< 3st$, which gives $|1-a_i\gamma| < 1/(12st)$.
Proceeding in the same way in all cases, we obtain
\begin{equation}\label{eq:LBparameters}
\begin{aligned}
a_i\geq 0 &\implies |1-a_i\gamma|<\frac{1}{12st}\implies a_i\gamma>1-\frac{1}{12st},\\
a_i<0 &\implies |1-a_i\delta|<\frac{1}{12st}\implies a_i\delta>1-\frac{1}{12st},\\
b_j\geq 0 &\implies |1-b_j\alpha|<\frac{1}{12st}\implies b_j\alpha>1-\frac{1}{12st},\\
b_j<0 &\implies |1-b_j\beta|<\frac{1}{12st}\implies b_j\beta>1-\frac{1}{12st}.
\end{aligned}
\end{equation}

\medskip
\noindent\textbf{Step 3: the top-left block controls the disagreements between $M$ and $uv^\top$.}~We now turn to the block $I\times J$.
We distinguish the two types of disagreements
between $M$ and $uv^\top$, together with the agreeing positive entries:
\begin{eqnarray*}
F_0&=&\{(i,j)\in I\times J:\ M_{i,j}=-1,\ u_iv_j=1\},\\
F_1&=&\{(i,j)\in I\times J:\ M_{i,j}=1,\ u_iv_j=-1\},\\
F_2&=&\{(i,j)\in I\times J:\ M_{i,j}=1,\ u_iv_j=1\}.
\end{eqnarray*}
The sets $F_0$ and $F_1$ are exactly the disagreement positions between $M$ and $uv^\top$,
hence
\begin{equation}\label{eq:disagreement_inner_product}
\langle M,uv^\top\rangle = st-2(|F_0|+|F_1|).
\end{equation}

We now compute the contribution of the block $I\times J$.
Every pair $(i,j)\in F_1$ contributes exactly $1$ to the objective, every pair $(i,j)\in F_0$ contributes $(a_ib_j)^2$, and every pair $(i,j)\in F_2$ contributes $(1-a_ib_j)^2$.
Therefore,
\begin{equation}\label{eq:EIJ_formula_final}
E_{IJ} = |F_1| +\sum_{(i,j)\in F_0}(a_ib_j)^2
+\sum_{(i,j)\in F_2}(1-a_ib_j)^2
\geq |F_1|+\sum_{(i,j)\in F_0}(a_ib_j)^2.
\end{equation}

It remains to bound $(a_ib_j)^2$ from below on $F_0$.
Suppose first that $a_i\ge0$ and $b_j\ge0$.
Using \eqref{eq:parameters_interval} and \eqref{eq:LBparameters}, we obtain
\[ (a_ib_j)^2 = \left(\frac{(a_i\gamma)(b_j\alpha)}{\alpha\gamma}\right)^2 > \left(\frac{\left(1-\frac{1}{12st}\right)^2}{1+\frac{1}{12st}}
\right)^2
\geq \left(1-\frac{3}{12st}\right)^2 = \left(1-\frac{1}{4st}\right)^2 > 1-\frac{1}{2st}, \]
since $(1-x)^2\geq (1-3x)(1+x)$ for $x\geq 0$. The case $a_i<0$ and $b_j<0$ is identical,
using instead $(a_i\delta)(b_j\beta)$ and $\beta\delta$.
Summing over all $(i,j)\in F_0$, and using $|F_0|\le st$, we get
\[ \sum_{(i,j)\in F_0}(a_ib_j)^2 > |F_0|-\frac{1}{2}. \]
Combining this with \eqref{eq:EIJ_formula_final} gives
\begin{equation}\label{eq:EIJ_clean_lower_bound}
E_{IJ}>|F_1|+|F_0|-\frac{1}{2}.
\end{equation}

\medskip
\noindent\textbf{Step 4: conclusion.}~From the previous steps, we know that $E_{IR}\geq Ns$, $E_{CJ}\geq Nt$, $E_{\mathrm BR}\geq 0$, and using \eqref{eq:EIJ_clean_lower_bound}, 
\[
\|X-\max(0,wh^\top)\|_F^2 = E_{IJ}+E_{IR}+E_{CJ}+E_{\mathrm BR} > N(s+t)+|F_1|+|F_0|-\frac{1}{2}. 
\]
On the other hand,
\[ \|X-\max(0,wh^\top)\|_F^2 \leq T = N(s+t)+\frac{st-D}{2}. \]
Therefore,
\[ |F_1|+|F_0| < \frac{st-D}{2}+\frac{1}{2}. \]
Since $|F_1|+|F_0|$ is an integer, the strict inequality above implies that
\[
|F_1|+|F_0| \leq \frac{st-D}{2}.
\]
Using \eqref{eq:disagreement_inner_product}, we conclude that
\[ \langle M,uv^T\rangle = st-2(|F_1|+|F_0|) \geq D. \]
Hence $(M,D)$ is a yes-instance of \textsc{Cut-Norm}.

\end{proof}

\section{Exact RMD: Geometry and ReLU-rank}
\label{sec:exact_NMD}

In this section, we investigate the exact RMD problem, which consists of finding two factors $W \in \mathbb{R}^{m \times r }$ and $H \in \mathbb{R}^{r \times n}$ such that $X=\max(0,WH)$, where $X$ is a given sparse and nonnegative matrix.
Among all possible solutions, we are interested in finding one that has the minimum possible rank, which is the $\rankrel(X)$ in Definition~\ref{def:f_rank}. Exact RMD is closely related to sign decomposition, which approximates a matrix with entries in $\{-1,1\}$, called a sign matrix, using the nonlinear model $\sgn(WH)$.
In the next subsection, we briefly review some properties of the sign decomposition, which serve as an introduction to our results on exact RMD.

\subsection{Sign decomposition:  geometric interpretation and bounds}

Sign decomposition finds application in a variety of contexts, such as communication complexity~\cite{paturi1986probabilistic,chattopadhyay2018separation,sherstov2011unbounded}, learning theory~\cite{ben2002limitations,klivans2007unconditional,feldman2017general}, and discrete geometry~\cite{alon1985geometrical,fox2010overlap}. Of particular interest is the sign-rank, defined as follows. 
\begin{definition}[Sign rank]
Let $S \in \{-1,1\}^{m \times n}$ be a sign matrix.
The sign-rank of $S$, denoted $\ranksg(S)$, is the smallest rank of a matrix $B \in \mathbb{R}^{m \times n}$ with nonzero entries such that $\sgn(B) = S$.
\end{definition}
Finding tight lower and upper bounds on the sign rank of a sign matrix is an active area of research in discrete linear algebra, see~\cite{alon2016sign,hatami2022lower}.  

The notion of $\ranksg$ and the geometric interpretation of the sign decomposition were introduced by Paturi and Simon~\cite{paturi1986probabilistic} in the context of communication complexity, and by Delsarte and Kamp~\cite{delsarte1989low} from the linear algebra point of view.
We briefly present the findings of Delsarte and Kamp, which focus on understanding which sign vector $s$ can be obtained from $\sgn(Wh)$ when $W$ is fixed.
Let $s \in \{-1,1\}^m$ be a sign vector.
We say that $s$ is sign-realized by $W \in \mathbb{R}^{m\times r}$ if there exists $h \in \mathbb{R}^r$ such that $s=\sgn(Wh)$.
Similarly, if $S$ is a $m$-by-$n$ sign matrix, it is sign-realized by  $W \in \mathbb{R}^{m\times r}$ if there exists  $H \in \mathbb{R}^{r\times n}$ such that $S=\sgn(WH)$.
Since it facilitates the geometric interpretation of RMD, we follow the construction from~\cite{delsarte1989low}, which considers the matrix $W$ in affine form, that is,
\begin{equation}
    W=\begin{pmatrix}
       1& \xi_1^\top\\
       \vdots & \vdots \\
       1 & \xi_m^\top
    \end{pmatrix}, \quad \text{where } \xi_i \in \mathbb{R}^{r-1}.
    \label{eq:struct_W}
\end{equation}
A decomposition where $W$ is as in~\eqref{eq:struct_W} is referred to as an affine decomposition.
For consistency, whenever the product $Wh$ is used with $W$ in affine form, we write $h=(h_0,\Bar{h})$ with $h_0 \in \mathbb{R}$ and $\Bar{h} \in \mathbb{R}^{r-1}$.
If each row of $W$ is viewed as a point in $\mathbb{R}^r$, an affine decomposition places all these points on the hyperplane $\{x\in\mathbb{R}^r : x_1=1\}$, which reduces the dimension of the representation by one.

For any sign vector $s$ such that $s=\sgn(Wh)$ with $W$ in affine form, we define the associated partition of the rows of $W$ as, 
  \begin{equation}
     \mathcal{P}_+=\{\xi_i: s_i=1\},  \qquad\mathcal{P}_-=\{\xi_i: s_i=-1\}.
     \label{eq:p+_p-_w}
\end{equation}
Using the same notation introduced in~\eqref{eq:struct_W}, the product $Wh$ can be represented by a collection of points  $\xi_i^\top$ lying in $\mathbb{R}^{r-1}$ and an affine hyperplane $\{ x \in \mathbb{R}^{r-1} \ | \ \Bar{h}^Tx=-h_0 \}$.
Let $s$ be a sign vector and $\mathcal{P}_+$ and $\mathcal{P}_-$ the associated partition.
Then $s=\sgn(Wh)$ with $W$ in affine form, if and only if the hyperplane with equation $\Bar{h}^Tx=-h_0$ separates $\mathcal{P}_+$ from $\mathcal{P}_-$ in $\mathbb{R}^{r-1}$.
Exploiting this geometric intuition, Delsarte and Kamp~\cite{delsarte1989low} proved a condition for a sign vector to be sign realized by a matrix $W$ assumed to be in affine form. 

\begin{lemma}[Theorem 2,~\cite{delsarte1989low}]
Let $s$ be a sign vector.
It is sign-realizable by the affine matrix $W$ if and only if the associated partition $\mathcal{P}_+$ and $\mathcal{P}_-$ of the rows of $W$ is such that the convex hulls of $\mathcal{P}_+$ and $\mathcal{P}_-$ are disjoint. 
\label{th:disj_conv_hull}
\end{lemma}
When considering a sign matrix instead of a sign vector, the sign decomposition $S=\sgn(WH)$ can be interpreted as a points-hyperplanes system in which the rows of $W$ define a collection of points in $\mathbb{R}^{r-1}$ and the columns of $H$ identify affine hyperplanes, which separate the points according to the signs of $S$. 
As an immediate consequence of Lemma~\ref{th:disj_conv_hull}, we have the following corollary.

\begin{corollary}[\cite{delsarte1989low}]
Let $S \in \{+,-\}^{m \times m}$ have positive signs on the diagonal and negative signs elsewhere.
Then $\ranksg(S) \leq 3$.
\label{th:id_sign}
\end{corollary}
The proof of Corollary~\ref{th:id_sign} becomes intuitive by using Lemma~\ref{th:disj_conv_hull}: each column of a diagonal sign matrix is a sign vector inducing a partition $\mathcal{P}_+$ and $\mathcal{P}_-$ of the rows of $W$ such that each  $\mathcal{P}_+$ contains a single row and $\mathcal{P}_-$ all the others.
Therefore, we need to find a set of points in the plane, representing the rows of $W$ in affine form, such that no point lies in the convex hull of the others.
For example, this property is satisfied by choosing $m$ disjoint points in a two dimensional circle. 

Moreover, Corollary~\ref{th:id_sign} can be generalized to the class of block diagonal matrices of the form
\begin{equation}
S=\begin{pmatrix}
    S^{(1)} & - & \cdots & -\\
    - & S^{(2)} & \cdots & -\\
    \vdots & \vdots& \ddots & \vdots \\
    - & - & & S^{(q)}
\end{pmatrix},
\label{eq:pat_id_sign_matrix}
\end{equation}
where each $S^{(i)}$, $i=1,\dots,q$ is a sign matrix of dimension $m_i \times n_i.$
\begin{lemma}[Theorem 4,~\cite{delsarte1989low}]
Let $S$ be a matrix as in~\eqref{eq:pat_id_sign_matrix} with blocks of the same dimensions.
Assume that each $S^{(i)}$ admits an affine sign decomposition of rank at most $k$.
Then $S$ has an affine type sign decomposition of rank not exceeding $k+1$. 
\label{th:block_diag_sign}
\end{lemma}

Lemma~\ref{th:block_diag_sign} is surprising since the bound on the rank does not depend on the number of blocks.
The proof of Lemma~\ref{th:block_diag_sign} along with more results on sign decomposition of sign matrices can be found in~\cite{delsarte1989low}. 
We now extend the theory of exact sign decomposition to the case in which the nonlinearity is given by the ReLU function.


\subsection{ReLU nonlinearity: geometric interpretation and bounds on the ReLU-rank}

A first lower bound on the $\rankrel$ of a given matrix $X$ was proposed recently in~\cite{saul2026subzero}: it is given by the rank of any submatrix containing only positive entries, that is, for any subsets $I$ and $J$ of rows and columns, respectively, such that $X(I,J)>0$, we have
\begin{equation}
\rankrel(X) \geq \rank(X(I,J)). 
\label{eq:saul_bound}
\end{equation}
The bound follows from the fact that any exact RMD of $X$ needs to preserve all its positive entries, including the positive submatrices.
We propose another lower bound that is a straightforward consequence of the connection between RMD and sign decomposition. 
\begin{definition}
Let $X \in \mathbb{R}^{m \times n}$ be a nonnegative matrix.
We denote by $S_X$ the matrix containing a $1$ where $X$ is positive and a $-1$ where $X$ is zero.
We refer to $S_X$ as the sign pattern of $X$. 
\end{definition}
The following lemma connects the $\rankrel$ of $X$ with the $\ranksg$ of its sign pattern $S_X$. 

\begin{lemma}
Let $X \in \mathbb{R}^{m \times n}$ be a nonnegative matrix and $S_X$ its sign pattern.
Then, 
\begin{equation}
\rankrel(X) \geq \ranksg(S_X) - 1 .
\label{eq:rank_bound}
\end{equation}
\label{lem:RMD_sign_patt}
\end{lemma}
\begin{proof}
Let $X=\max(0,WH)$ be an RMD of $X$ of rank $ \rankrel(X)$, we distinguish two cases. First, assume $(WH)_{i,j}\neq0$ for every $i=1,\dots,m$ and $j=1,\dots,n$. Then any exact decomposition $X=\max(0,WH)$ must also satisfy $S_X=\sgn(WH)$, which implies $\rankrel(X) \geq \ranksg(S_X)$.

If instead $(WH)_{i,j}=0$ for some $(i,j)$, then let $\mathcal{Q}=\{(i,j) : X_{i,j}>0\}$, fix $\delta=\min_{ (i,j) \in \mathcal{Q} } X_{i,j}$, and take $\varepsilon \in (0, \delta)$. Consider $E = WH - \varepsilon \mathbf{1}_m\mathbf{1}_n^T$, which has rank at most $ \rankrel(X)+1$ and $\sgn(E)=S_X$; thus
\begin{equation*}
\ranksg(S_X) \leq \rank(E) \leq \rankrel(X)+1.
\end{equation*}
\end{proof}
Lemma \ref{lem:RMD_sign_patt} shows that any exact RMD yields an exact sign decomposition for the corresponding sign pattern. Note that the converse is not true since sign decomposition does not preserve the positive entries of $X$.
However, the gap between the two ranks might be large.
If $X$ is positive and full-rank, then $\ranksg(S_X)=1$ while $\rankrel(X)=\min(m,n)$. Conversely, if $X$ is binary, every positive submatrix has rank one, so~\eqref{eq:saul_bound} gives 1, whereas $\ranksg(S_X)$ can be large.
Hence there is no hierarchy between the bounds in~\eqref{eq:saul_bound} and~\eqref{eq:rank_bound}, their values depend on the sparsity and the distribution of the zeros in the matrix.

\paragraph{Geometric intuition} Similarly to the sign decomposition case, we say that a vector $y \in \mathbb{R}_+^m$ (a matrix $Y \in \mathbb{R}_+^{m \times n} $) is ReLU-realized by a matrix $W$ if there exists a vector $h$ (a matrix $H \in \mathbb{R}^{r \times n} $) such that $y=\max(0,Wh)$ ($Y=\max(0,WH)$).
We now address the problem of understanding which vectors can be ReLU-realized by a fixed matrix $W$, which is assumed to be in affine form as in~\eqref{eq:struct_W}.
In order to do so, we define the signed distance of a point $\xi$ from a hyperplane $\gamma=\{ x \in \mathbb{R}^{r-1} \ | \  \Bar{h}^\top x=-h_0 \}$ with $\Bar{h} \neq 0 $ as
\begin{equation}
     d(\gamma,\xi)=\frac{h_0+\xi^\top \Bar{h}}{\lVert \Bar{h} \rVert}. 
     \label{eq:signed_dist}
\end{equation}

\begin{lemma}
    Let $y$ be a nonnegative vector such that $y\neq c\,\mathbf{1}$ for every
    $c>0$, and $W$ be in affine form with rows $(1,\xi_i^\top)$. Let also $\mathcal{P}_+=\{\xi_i: y_i>0\}$
    and $\mathcal{P}_-=\{\xi_i: y_i=0\}$ be the partition
    induced by the sign pattern of $y$. Then $y$ is ReLU-realizable by $W$ as
    in~\eqref{eq:struct_W} if and only if there exists a hyperplane
    $\gamma=\{x\in\mathbb{R}^{r-1} \mid \bar h^\top x=-h_0\}$, with
    $\bar h\neq 0$, that separates the convex hulls of
    $\mathcal{P}_+$ and $\mathcal{P}_-$, and such that
    \begin{equation}
        d(\gamma,\xi_i)\,\lVert \bar h\rVert = \xi_i^\top \bar h + h_0 = y_i
        \quad \text{for every } i \text{ such that } y_i>0.
        \label{eq:cond_relu}
    \end{equation}
    \label{lem:ext_conv_hull}
\end{lemma}

\begin{proof}
Let $y$ be such that $y=\max\Big(0,W \begin{pmatrix}
    \hat{h}_0 \\
    \hat{h}
\end{pmatrix}\Big)$, with $W$ in affine form and $\hat{h}\neq0$, which implies $y_i=\xi_i^T \Hat{h}+\Hat{h}_0 >0$ for $i$ such that $y_i>0$ and $y_i=\xi_i^T \Hat{h}+\Hat{h}_0 \leq0$ for $i$ such that $y_i=0$. Then the hyperplane \newline \mbox{$\Bar{\gamma}=\{ x \in \mathbb{R}^{r-1} \ | \  \hat{h}^\top x=-\hat{h}_0 \}$} separates the convex hull of $\mathcal{P}_+$ and $\mathcal{P}_-$. Moreover, by the definition of signed distance in~\eqref{eq:signed_dist}, it also satisfies~\eqref{eq:cond_relu}. Conversely, let \mbox{$\gamma= \{ x \in \mathbb{R}^{r-1} \ | \  \Bar{h}^\top x=-h_0 \}$}, with $\Bar{h}$ non zero, separating $\mathcal{P}_+$ and $\mathcal{P}_-$ and satisfying~\eqref{eq:cond_relu}. Using~\eqref{eq:cond_relu}, we have $y_i=h_0+\xi_i^\top \Bar{h}=(Wh)_i$ for every $i  $ such that $y_i>0$, with $h=(h_0,\Bar{h})$. Moreover, $\gamma$ separates $\mathcal{P}_+$ and $\mathcal{P}_-$; therefore, it must be $(Wh)_i \leq 0$ for every $i  $ such that $y_i=0$. Combining these two observations, we obtain $y=\max(0,Wh)$.
\end{proof}
Lemma~\ref{lem:ext_conv_hull} does not cover the case where $y=c\,\mathbf{1}$ with $c>0$, which corresponds to a degenerate hyperplane with $\Bar{h}=0$. However, in that case $y$ is always ReLU-realizable by $W$, by choosing $h_0=c$ and $\bar h=0$. 

Lemma~\ref{lem:ext_conv_hull} shows that the points-hyperplanes interpretation introduced for the sign decomposition extends to RMD.
Nevertheless, RMD imposes a relationship between the positive entries of the original matrix and the distance of the hyperplane to the points associated with these entries in the RMD representation. Therefore, it results in a more restrictive points-hyperplanes system. 
In the exact RMD case $X=\max(0,WH)$, the dependence on $\lVert\bar h\rVert$ in~\eqref{eq:cond_relu} cannot be removed: normalizing the column of $H$ can only be achieved by multiplying $\max(0,WH)$ on the right with a positive diagonal matrix and thus cannot be compensated without modifying the values in the original matrix $X$.

\paragraph{Bounds on sparse, structured matrices} 

We now use the geometric interpretation of RMD to provide some bounds on the $\rankrel$ of special classes of matrices, namely 
diagonal matrices (Theorem~\ref{th:relu_id}), 
block diagonal matrices (Theorem~\ref{th:block_id_relu_aff}), 
Kronecker product of two matrices (Theorem~\ref{th:kron}), 
and sum of Kronecker product of two matrices (Theorem~\ref{th:sumKrosparse}).

The first example is the class of diagonal matrices, for which we derive a bound by extending the result on the $\ranksg$ of diagonal sign matrices in Corollary~\ref{th:id_sign}.
This example was already considered in~\cite{saul2022nonlinear,saul2026subzero}, where a constructive proof is proposed.
Our proof is much shorter and straightforward as it follows directly from the connection with the sign decomposition.

\begin{theorem}
    Let $X=\diag(d_1, \dots, d_m)$ be a diagonal, nonnegative matrix with positive diagonal elements $d_i \geq 0$. Then, 
    \begin{equation}
        \rankrel(X) \leq 3,
    \end{equation}
    \label{th:relu_id}
    and $X$ admits an exact rank-3 RMD in affine form. 
\end{theorem}
\begin{proof}
    From Corollary~\ref{th:id_sign} we know that there exists an $m \times 3$ matrix $W$ that can realize the sign matrix containing 1 on the diagonal and -1 elsewhere, which also represents the sign pattern of $X$.
    Therefore, for every $i=1,\dots,m$ there exists a hyperplane $\gamma_i: \{x \in \mathbb{R}^2 \ | \  \Bar{h}_{i}x=-h_{0,i}\}$ separating $\xi_i$ from all $\xi_j$ for $j \neq i$.
    Thus, it holds $\xi_i^\top \Bar{h}_i + h_{0,i}=\alpha_i >0$ and $\xi_j^T \Bar{h}_i + h_{0,i}<0$ for every $j \neq i.$
    Let $H$ be the matrix containing all the $(h_{0,i},\Bar{h}_i)$ as columns.
    Multiplying each column by a factor $\frac{d_i}{\alpha_i}$, we obtain a new matrix $\Bar{H}$ such that
    \begin{equation}
        W\Bar{H}=W \left[ \frac{d_1}{\alpha_1}\Bar{h}_1, \dots, \frac{d_m}{\alpha_m}\Bar{h}_m \right] = \begin{pmatrix}
             d_1 & - & \cdots & -\\
    - & d_2 & \cdots & -\\
    \vdots & \vdots& \ddots & \vdots \\
    - & - & & d_m
        \end{pmatrix};
        \label{eq:ReLU_theta_id}
    \end{equation}
   thus, proving the theorem. 
\end{proof}

Note that rescaling the matrix $H$ is possible because only one positive element per column needs to be exactly matched by the RMD; thus, there is no ambiguity in how to rescale the columns of $H$. 

Let us generalize Theorem~\ref{th:relu_id} to block diagonal matrices of the form
\begin{equation}
X=\begin{pmatrix}
    X^{(1)} & 0 & \cdots & 0\\
    0 & X^{(2)} & \cdots & 0\\
    \vdots & \vdots& \ddots & \vdots \\
    0 &0 & \cdots & X^{(q)}
\end{pmatrix}.
\label{eq:pat_id_relu_matrix}
\end{equation}
Let $\bar{m}$ and $\bar{n}$ denote the maximum row and column dimensions, respectively, among all the blocks. Saul et al. proved in~\cite[Lemma 1]{saul2026subzero} that $\rankrel(X)\leq 3\min(\Bar{m},\Bar{n})$. Following the intuition from Lemma~\ref{th:block_diag_sign}, we show in what follows that, if each block admits an affine RMD of rank $r_k$, then $\rankrel(X)\leq 1+\max_k r_k \leq 1+\min(\Bar{m},\Bar{n})$, which improves the known bound on the $\rankrel$ of block diagonal matrices.

\begin{theorem}
   Let $X$ be as in (\ref{eq:pat_id_relu_matrix}), such that each block has dimensions $m_k \times n_k$ and $\rankrel(X^{(k)})=r_k \geq 2$, with $X^{(k)}=\max(0,\Bar{W}^{(k)} \Bar{H}^{(k)})$ and $\Bar{W}^{(k)}$ in affine form. Then, $X$ admits an exact RMD in affine form and
   \begin{equation}
      \rankrel(X)\leq\max_k r_k +1.
    \end{equation}
    \label{th:block_id_relu_aff}
\end{theorem}

\begin{proof}
    Let $\Bar{r}=\max_k r_k$.
    To match the dimensions of each block, we append zero columns to $\Bar{W}^{(k)}$ and zero rows to $\Bar{H}^{(k)}$ whenever $r_k < \Bar{r}$.
    This does not affect the decomposition and thus, we do not rename the zero padded versions of the matrices $\Bar{W}^{(k)}$, and $\Bar{H}^{(k)}$. Since each $\Bar{W}^{(k)}$ is assumed to be in affine form, all the blocks admit a rank-$\Bar{r}$ RMD that can be geometrically represented as a point-hyperplane configuration in $\mathbb{R}^{\Bar{r}-1}$.
    Consequently, each block representation can be embedded in an $(\Bar{r}-1)$-dimensional hyperplane in $\mathbb{R}^{\Bar{r}}$.
    Then, consider a ball in $\mathbb{R}^{\Bar{r}}$, centered at the origin, of radius $R$, and $q$ distinct points $c_1, \dots, c_q$ on its surface. The idea of the proof is to arrange each embedded block representation in the hyperplanes $\mathcal{D}^k=\{ x \in \mathbb{R}^{\Bar{r}} \ | \ v_k^\top x=R \}$ that are tangent to the ball in $c_k$, with $ v_k=c_k/R $ unit vector.
    Note that such a construction can be obtained by affine transformations of the points (rows of $\Bar{W}^{(k)}$) and hyperplane coefficients (columns of $\Bar{H}^{(k)}$) in the points-hyperplanes configuration defined by the RMD of each block.
    Therefore, we use the following notation for the embedded points-hyperplanes configuration:
    \begin{itemize}
        \item $w_i^k \in \mathbb{R}^{ \Bar{r}}$ for $i=1,\dots,m_k$ denotes the embedded points.
        \item $a_j^k \in \mathbb{R}$ and $z_j^k \in \mathbb{R}^{\Bar{r}}$ for $k=1,\dots,q$, and $j=1,\dots,n_k$, the coefficients of the embedded $(\Bar{r}-2)$-dimensional affine subspace in $ \mathcal{D}^k$, that is,
        $$\mathcal{L}_j^k=\left\{x \in \mathcal{D}^k  \ | \ (z_j^k)^\top x=-a_j^k \right\}=\left\{x \in \mathbb{R}^{\Bar{r}} \ | \ (z_j^k)^\top x=-a_j^k, \  v_k^\top x = R\right\}.$$
    \end{itemize}
    By construction, $z_j^k$ is a tangent direction in $\mathcal{D}^k$, hence $v_k^\top z_j^k=0$ for $j=1,\dots,n_k$.
    Let $(\mathcal P_+)^k_j$ and $(\mathcal P_-)^k_j$ denote the partition of the points $w_i^k$ induced by the sign pattern of the $j$-th column of $X^{(k)}$. Since the construction preserves inner products, we have
    \begin{equation}
        (w_i^k)^\top z^k_j+a_j^k=X^{(k)}_{i,j} \quad \text{for every } i \in (\mathcal{P}_+)^k_j \qquad \text{and} \qquad   (w_i^k)^\top z^k_j+a_j^k \leq 0 \quad \text{for every } i \in (\mathcal{P}_-)^k_j.
        \label{eq:cond_W_bar}
    \end{equation}
 Moreover, we choose $R$ large enough so that, for every $k$, all points $w_i^s$ with $s\neq k$ lie on the same side of $\mathcal D^k$, at distance at least $\delta_1$ from it, meaning that $(w_i^s)^\top v_k = (w_i^s)^\top \frac{c_k}{R}  \leq  R -\delta_1 $ for all $s \neq k$ and $i=1,\dots,m_s$.
 Note that this is always possible since the points $w_i^s$ are fixed on the hyperplane $\mathcal{D}^s$ while the hyperplane offset grows linearly with $R$.
 
 Since our aim is to find an RMD of rank $\Bar{r}+1$, for every $k=1,\dots q$ and $j=1,\dots,n_k$, we need to define an affine hyperplane $\mathcal{H}^k_j=\{ x \in \mathbb{R}^{\Bar{r}} \ | \ (\Bar{h}_j^k)^\top x = -b_j^k\}$ that preserves the properties in~\eqref{eq:cond_W_bar} for $(\mathcal{P}_+)^k_j$, and $(\mathcal{P}_-)^k_j$ and also separates  $(\mathcal{P}_+)^k_j$ from all points in the other blocks, that is, 
  \begin{equation}
  \begin{aligned}
      & (1,w_i^k) (b_j^k,\Bar{h}^{k}_j)^\top=X^{(k)}_{i,j} \quad \text{for every } i \in (\mathcal{P}_+)^k_j, \\
      &(1,w_i^k) (b_j^k,\Bar{h}^{k}_j)^\top \leq 0 \quad \text{for every } i \in (\mathcal{P}_-)^k_j \cup \{ w_t^s \ | \ s=1,\dots,q \  s \neq k, \ t=1,\dots,m_s\}.
  \end{aligned}
  \label{eq:final_cond_diag}
    \end{equation}
It remains to show how to construct such a hyperplane.
First, define the hyperplane $\mathcal{S}^k_j$, passing through $\mathcal{L}_j^k$ and the origin, that is, $\mathcal{S}^k_j=\{ x \in \mathbb{R}^{\Bar{r}} \ | \ (a_j^k v_k+R z_j^k)^\top x=0 \}$.
Second, we rotate $\mathcal{S}_j^k$ by an angle $\theta_j$ around the fixed axes defined by the subspace $\mathcal{L}^k_j$.
Doing so is equivalent to applying the standard two-dimensional rotation to the normal vector $a_j^k v_k+R z_j^k$ of $\mathcal{S}_j^k$ in the plane
$\Span\{v_k,z_j^k / \lVert z_j^k  \rVert\}.$
The offset $a_j^k(\theta_j)$ is uniquely determined by requiring that the rotated hyperplane $\mathcal{S}_j^k(\theta_j)$ still contains $\mathcal{L}_j^k$, hence obtaining
\begin{equation*}
  \mathcal{S}_j^k(\theta_j)= \left\{ x \in \mathbb{R}^{\Bar{r}} \ | \  z_j^k(\theta_j)^\top x = -a_j^k(\theta_j) \right\},
\end{equation*}
with
\begin{equation*}
      z_j^k(\theta_j)=(a_j^k\cos(\theta_j)-R\lVert z_j^k \rVert  \sin(\theta_j)) v_k + \left(\frac{a_j^k}{\lVert z_j^k \rVert }\sin(\theta_j)+R\cos(\theta_j)\right)z_j^k ,
      \end{equation*} 
      and 
      \begin{equation*} 
      a_j^k(\theta_j)=\left( R^2 \lVert z_j^k \rVert + \frac{(a_j^k)^2}{\lVert z_j^k \rVert} \right) \sin(\theta_j).
\end{equation*}
Taking the scalar product with $w_i^k \in \mathcal{D}^k$ and using $(w_i^k)^Tv_k=R$,  we have
\begin{equation}
    \begin{aligned}
      (w_i^k)^\top z_j^k(\theta_j) + a_j^k(\theta_j) =& (a_j^k\cos(\theta_j)-R\lVert z_j^k \rVert  \sin(\theta_j))R + \left(\frac{a_j^k}{\lVert z_j^k \rVert }\sin(\theta_j)  +R\cos(\theta_j)\right) (w_i^k)^\top z_j^k  +\\
      & \left( R^2 \lVert z_j^k \rVert +\frac{(a_j^k)^2}{\lVert z_j^k \rVert} \right) \sin(\theta_j)  \\
      =& a_j^k R \cos(\theta_j) +  \left(\frac{a_j^k}{\lVert z_j^k \rVert }\sin(\theta_j)+R\cos(\theta_j)\right) (w_i^k)^\top z_j^k + \frac{(a_j^k)^2}{\lVert z_j^k \rVert} \sin(\theta_j) \\
       =& \left( R\cos(\theta_j) + \frac{a_j^k}{\lVert z_j^k \rVert }\sin(\theta_j)  \right) ((w_i^k)^\top z_j^k + a_j^k)=d_j(\theta_j) ((w_i^k)^\top z_j^k + a_j^k).
    \end{aligned}
    \label{eq:long_eq_theta}
\end{equation}
Therefore, the relation among the inner product is preserved up to the factor $d_j(\theta_j)$, which is the same for all points $w_i^k \in \mathcal{D}^k$.
Since we want the sign of the scalar products to be preserved, we choose an angle $\theta_j$ such that $d_j(\theta_j)$ is positive.
Solving $d_j(\theta_j) > 0$, we obtain
\begin{equation*}
    \theta_j \in \left]\, \theta_{\min},\ \theta_{\max} \,\right[,
    \qquad
    \theta_{\min} = \arctan\left(\frac{a_j^k}{R \lVert z_j^k \rVert}\right) - \frac{\pi}{2},
    \quad
    \theta_{\max} = \arctan\left(\frac{a_j^k}{R \lVert z_j^k \rVert}\right) + \frac{\pi}{2}.
\end{equation*}
Both endpoints of this interval make $\mathcal{S}_j^k(\theta_j)$ collapse onto
$\mathcal{D}^k$, with opposite orientations.
One can verify that as $\theta_j \to \theta_{\min}$, the orientations of $\mathcal{S}_j^k(\theta_j)$ and $\mathcal{D}^k$ coincide. 
Moreover, we observe that the function of $\theta_j$ that brings $ \mathcal{S}_j^k(\theta_j)$ in $ \mathcal{D}^k$ or equivalently $z_j^k(\theta_j)$ into $v_k$ and $-a_j^k(\theta_j)$ into $R$, is continuous for every $j=1,\dots,n_k$. Combining this fact with~\eqref{eq:long_eq_theta}, we deduce that, for every $j=1,\dots,n_k$, there exists an angle $\theta_j^* \in ]\theta_{\min},\theta_{\max}[$ such that it also holds
\begin{equation}
   (w_i^s)^T z_j^k(\theta_j^*) + a_j^k(\theta_j^*) \leq  0 \quad \text{for every }  i=1,\dots,m_s, \ s=1,\dots,q, \text{with } s \neq k,
   \label{eq:outdiag}
\end{equation}
hence, we found a hyperplane satisfying the conditions in~\eqref{eq:final_cond_diag} with $\Bar{h}_j^k=z_j^k(\theta_j^*)$ and $b_j^k=a_j^k(\theta_j^*)$.

Extending the same reasoning to all the columns $j=1,\dots,n_k$, and defining \newline \mbox{$ D^{(k)}=\text{diag}(d_1(\theta_1^*),\dots,d_{n_k}(\theta_{n_k}^*))$}, we have that $X^{(k)} D^{(k)} = \max(0,W^{(k)}  H^{(k)})$, where the $i$-th row of $W^{(k)} \in \mathbb{R}^{m_k \times (\Bar{r}+1)} $ contains the vector $[1,(w_i^k)^T]$ and the $j$-th column of $H^{(k)} \in \mathbb{R}^{ (\Bar{r}+1) \times n_k} $ contains $[a_j^k(\theta_j^*),z_j^k(\theta_j^*)]$. Thus, we have
\begin{equation}
    X^{(k)}=\max(0,W^{(k)} H^{(k)}) (D^{(k)})^{-1}=\max(0,W^{(k)} H^{(k)} (D^{(k)})^{-1}),
    \label{eq:diag}
\end{equation}
where the last equality holds because $(D^{(k)})^{-1}$ is a positive diagonal matrix. Therefore, combining~\eqref{eq:outdiag} and~\eqref{eq:diag} we obtain  

\begin{equation*}
X = \max \left( 0,  \begin{pmatrix}
    W^{(1)} \\
    W^{(2)} \\
    \vdots \\
    
    W^{(q)}
\end{pmatrix}
 \begin{pmatrix}
    H^{(1)} (D^{(1)})^{-1} &  H^{(2)}  (D^{(2)})^{-1}  & \cdots &  H^{(q)} (D^{(q)})^{-1} \\
\end{pmatrix} 
\right);
\end{equation*}
hence proving the theorem.

\end{proof}

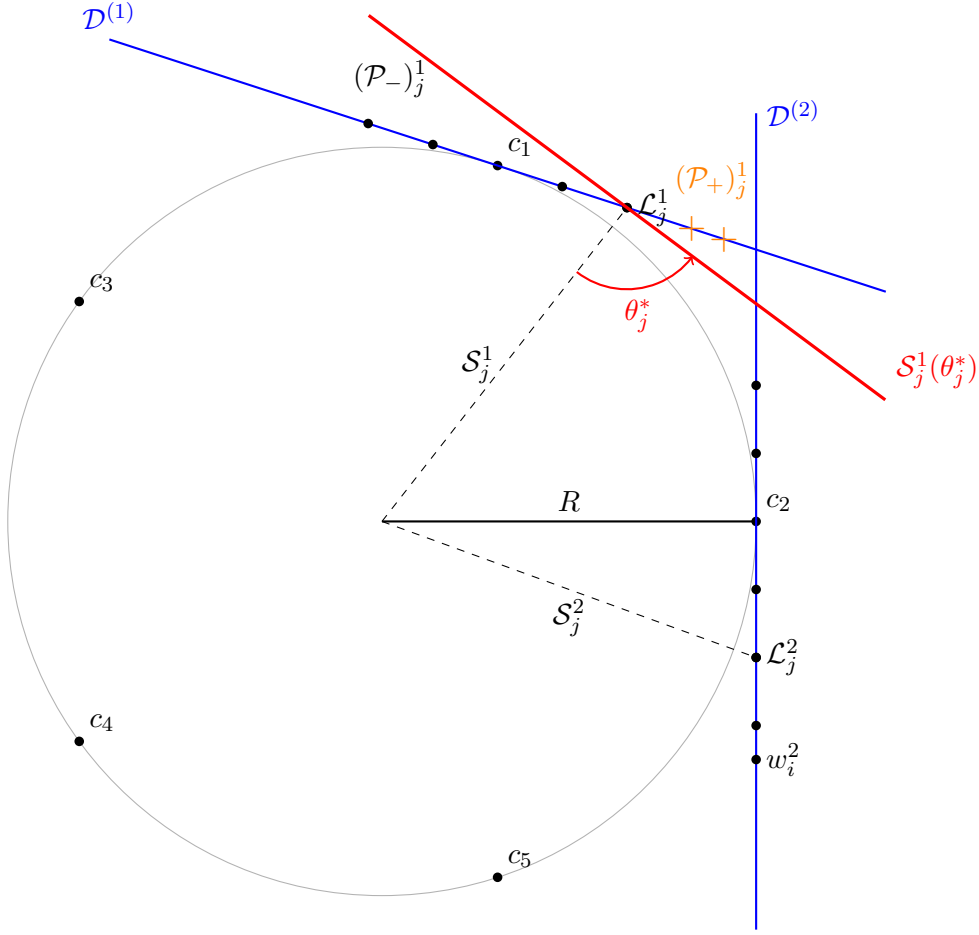
\begin{figure}[t]
\centering

\begin{tikzpicture}[scale=0.9]


\def\R{5.5}
\def\alpha{72}


\coordinate (O) at (0,0);


\draw[gray!60] (O) circle (\R);

\foreach \ang/\name in {
72/1,
0/2,
144/3,
216/4,
288/5}
{
    \coordinate (C\name) at ({\R*cos(\ang)},{\R*sin(\ang)});
    \fill (C\name) circle (2pt);
    \node[above right] at (C\name) {$c_{\name}$};
}

\draw[thick] (0,0) -- (C2) node[midway, above] {$R$};

\draw[blue, thick]
(\R,-6)--(\R,6)
node[right]{$\mathcal D^{(2)}$};

\foreach \y in {-3.5,-3,-2,-1,1,2}
{
    \fill (\R,\y) circle (2pt);
}

\node[right] at (\R,-3.5)
{$ w_i^2$};


\coordinate (L2) at (\R,-2);

\fill (L2) circle (2pt);

\node[right] at (L2)
{$\mathcal L_j^2$};

\draw[dashed]
(O)--(L2)
node[midway, below]
{$\mathcal S_j^2$};


\coordinate (C) at (C1);

\coordinate (tan) at ({-sin(\alpha)},{cos(\alpha)});

\draw[blue, thick]
($(C)-6*(tan)$)--($(C)+6*(tan)$)
node[above]
{$\mathcal D^{(1)}$};

\foreach \t in {-2,-1,1,2}
{
    \fill ($(C)+\t*(tan)$) circle (2pt);
}

\foreach \t in {-3.5,-3}
{
   \node[orange] at ($(C)+\t*(tan)$){\Large$+$};
}


\coordinate (Lj) at ($(C)-2*(tan)$);

\fill (Lj) circle (2pt);

\node[above, right]
at (Lj)
{$\mathcal L_j^1$};


\draw[dashed]
(O)--(Lj)
node[midway,left]
{$\mathcal S_j^1$};


\coordinate (B) at (\R,3.2);

\coordinate (rot_dir) at ($(B)-(Lj)$);

\draw[red, very thick]
($(Lj)-2*(rot_dir)$)
--
($(Lj)+2*(rot_dir)$)
node[above right]
{$\mathcal S_j^1(\theta_j^*)$};


\coordinate (Sdir) at ($(O)-(Lj)$);

\coordinate (Aangle) at ($(Lj)+1.5*(Sdir)$);
\coordinate (Bangle) at ($(Lj)+1.5*(rot_dir)$);

\pgfmathanglebetweenpoints
{\pgfpointanchor{Lj}{center}}
{\pgfpointanchor{Aangle}{center}}
\let\startangle\pgfmathresult

\pgfmathanglebetweenpoints
{\pgfpointanchor{Lj}{center}}
{\pgfpointanchor{Bangle}{center}}
\let\endangle\pgfmathresult

\draw[red, thick, ->]
(Lj)
++(\startangle:12mm)
arc[
    start angle=\startangle,
    end angle=\endangle,
    radius=12mm
]
node[midway, below, red]
{$\theta_j^*$};


\node[black]
at ($(C)+2.5*(tan)+(0.8,0.5)$)
{$(\mathcal P_-)_j^1$};

\node[orange]
at ($(C)-3*(tan)+(0.3,0.7)$)
{$(\mathcal P_+)_j^1$};

\end{tikzpicture}

\caption{Geometric construction of the proof of Theorem~\ref{th:block_id_relu_aff} in the two-dimensional case.
In red: the affine hyperplane $\mathcal{S}_j^1(\theta_j^*)$, obtained by rotating $\mathcal{S}_j^1$ by an angle $\theta_j^*$ around $\mathcal{L}_j^1$.
The new hyperplane separates the points in orange $(\mathcal{P}_+)_j^1$ from those in $(\mathcal{P}_-)_j^1 \cup \{w_i^2 \ | \ i=1,\dots,n_s \}$, with $n_s$ denoting the number of rows in $W^{(s)}$.}
\label{fig:geom_interpr_2D}

\end{figure}

We now derive a bound on the $\rankrel$ of the class of matrices obtained from the Kronecker product of a binary matrix in $\{0,1\}$ with a fixed matrix $A$.
\begin{theorem}   \label{th:kron} 
  Let $A \in \mathbb{R}_+^{m \times n}$ and $B \in \{0,1\}^{p \times q}$, assume $\rankrel(A)=r_A$ and $\rankrel(B)=r_B$. Then there exists an affine RMD of the Kronecker product matrix $B\otimes A$ of rank at most $r_A+r_B+1$, hence
  \begin{equation}
  \rankrel(B\otimes A) \leq r_A+r_B+1.
  \label{eq:kron_rank}
   \end{equation}
\end{theorem}
\begin{proof}
    Let $(U,V)$ be a rank-$r_B$ RMD of $B$, $B=\max(0,UV).$
    Denote the rows of $U$ by $u_i^\top$ and the columns of $V$ as $v_j$.
    Let $W$ and $H$ such that $A=\max(0,WH)$ and $\rank(WH)=r_A$, and let $\delta>0$ be a positive coefficient.
    Define
    \begin{equation}
    \Theta=\begin{pmatrix}
         \mathbf 1_m&\mathbf 1_m u_1^\top & W &  \\
        \vdots & \vdots & \vdots\\
         \mathbf 1_m& \mathbf 1_m u_p^\top& W  \\
    \end{pmatrix}
    \begin{pmatrix}
    -\delta\mathbf 1_n^\top & \cdots & -\delta \mathbf 1_n^\top\\
    \delta v_1 \mathbf 1_n^\top & \cdots & \delta v_q\mathbf 1_n^\top\\
        H & \cdots & H\\
    \end{pmatrix},
\end{equation}
where $\mathbf 1_m \in \mathbb{R}^m$ denotes the vector of all 1.
Then, $\Theta \in \mathbb{R}^{mp \times nq}$ has rank smaller than $r_A+r_B+1$, and we now show that it is an RMD of $B \otimes A$.
Since $\max(0,UV)$ is an RMD of the binary matrix $B$, it holds that $u_i^\top v_j=1$ whenever $B_{i,j}=1$ and $u_i^\top v_j \leq 0$ whenever  $B_{i,j}=0$.
Then, by construction, the matrix $\Theta$ contains a block of the form $W H + \delta (u_i^\top v_j-1)\mathbf{1}_m \mathbf{1}_n^\top=W H$ for all $(i,j)$ corresponding to entries for which $B_{i,j}=1$.
In contrast, when $B_{i,j}=0$, $u_i^\top v_j \leq 0$ and the corresponding block matrix in $\Theta$ is $W H + \delta  (u_i^\top v_j-1) \mathbf{1}_m \mathbf{1}_n^\top$.
Therefore, imposing $\delta > \lVert WH \rVert_\infty$, the block contains only negative values; thus, proving the theorem.
\end{proof}

Note that in general $\rank(B \otimes A) \leq \rank(A) \rank(B)$, while Lemma~\ref{th:kron} provides an upper bound for $\rankrel(B \otimes A)$ which is additive in $\rankrel(A)$ and $\rankrel(B)$. However, the bound in (\ref{eq:kron_rank}) is not tight in general. 
For example, when $B$ is the $2\times2$ identity matrix, we have 
\begin{equation}
    X=\begin{pmatrix}
        1 & 0\\
        0 & 1\\
    \end{pmatrix}\otimes A = \begin{pmatrix}
        A & 0\\
        0 & A\\
    \end{pmatrix}=\max \left( 0, \begin{pmatrix}
        A & -A\\
        -A & A\\
    \end{pmatrix}\right)=\max(0,\Theta). 
\end{equation}
This means that $\rankrel(X)=\rankrel(A)$, since the matrix on the right hand side has the same $\rankrel$ as $A$. Once again, this example is counterintuitive at first sight since the dimensions of $X$ are twice those of $A$. Similarly, consider
\begin{equation}
    X= B \otimes A=\begin{pmatrix}
        1 & 1\\
        1 & 0\\
    \end{pmatrix} \otimes A=\begin{pmatrix}
        A & A\\
        A & 0\\
    \end{pmatrix}.
\end{equation}
By Theorem~\ref{th:kron}, we have that $\rankrel(X) \leq \rankrel(A)+3$, but a tighter bound holds.
First, note that $\rankrel(B)=\rank(B)=2$, then write $A=\max(0,WH),$ and
\begin{equation}
    X=\begin{pmatrix}
        A & A\\
        A & 0\\
    \end{pmatrix}=\max \left(\begin{pmatrix}
        W & 0\\
        W & \Bar{a}\\
    \end{pmatrix} \begin{pmatrix}
        H & H\\
        0 & \Bar{b}^\top\\
    \end{pmatrix},0\right)= \max \left(\begin{pmatrix}
        WH & WH\\
        WH & WH+\Bar{a}\Bar{b}^\top\\
    \end{pmatrix},0 \right), 
    \label{eq:tri_A_block}
\end{equation}
where $\Bar{a}\in \mathbb{R}^m$, $\Bar{b} \in \mathbb{R}^n$, suitably chosen such that $WH+\Bar{a}\Bar{b}^\top$ contains only negative values.
Equation~\eqref{eq:tri_A_block} shows that $\rankrel(X) \leq \rankrel(A)+1.$ 

We now consider block-sparse matrices with different blocks denoted by $A^{(1)}, \dots,A^{(q)} $.
We are interested in matrices of the form 
\begin{equation}
    X=\sum_{k=1}^q  C^{(k)} \otimes A^{(k)},
    \label{eq:block_mat}
\end{equation}
where $C^{(k)} \in \{0,1\}^{p \times s}$ is a binary matrix, containing a single nonzero element at position $(i_k,j_k)$, that is,  $C^{(k)}=e_{i_k}e_{j_k}^\top$ where $e_{i}$ is the $i$-th unit vector.
Therefore, each block matrix can appear only once in $X$, and the position of the $k$-th block $A^{(k)}$ in $X$ corresponds to the position of the nonzero element in $C^{(k)}$. 
We also define $G=\sum_{k=1}^q C^{(k)} \in \{0,1\}^{p \times s}$, which indicates where blocks appear in $X$.   
For simplicity, we assume that the block matrices $A^{(k)}$ are all $m \times n$ while the matrices $C^{(k)}$ has dimensions $p \times s$.
Therefore, the final matrix $X$ has dimension $pm\times sn$. 
An example of a block-sparse matrix as in~\eqref{eq:block_mat} is  
\begin{equation}
    \begin{pmatrix}
        A^{(1)} & 0 & A^{(2)} \\
        0 & A^{(3)} & 0 \\
        0 & 0 & A^{(4)}
    \end{pmatrix} =  \underbrace{\begin{pmatrix}
        1 & 0 & 0 \\
        0 & 0 & 0 \\
        0 & 0 & 0
    \end{pmatrix}}_{ C^{(1)}} \otimes A^{(1)} + \underbrace{\begin{pmatrix}
        0 & 0 & 1 \\
        0 & 0 & 0 \\
        0 & 0 & 0
    \end{pmatrix}}_{ C^{(2)}} \otimes A^{(2)}+\underbrace{\begin{pmatrix}
        0 & 0 & 0 \\
        0 & 1 & 0 \\
        0 & 0 & 0
    \end{pmatrix}}_{ C^{(3)}} \otimes A^{(3)}+
    \underbrace{\begin{pmatrix}
        0 & 0 & 0 \\
        0 & 0 & 0 \\
        0 & 0 & 1
    \end{pmatrix}}_{ C^{(4)}} \otimes A^{(4)}.
    \label{eq:ex_block}
\end{equation}
We now introduce the definition of a \textit{block-diagonal submatrix} of $X$. 
\begin{definition}
    Let $X$ be defined as in~\eqref{eq:block_mat}, and let
    $P \subseteq \{1,\dots,q\}$. We say that the blocks 
   $\{ A^{(k)} \}_{k\in P}$  form a block diagonal submatrix of $X$ if, after a suitable permutation
    of rows and columns, the rows and columns of $X$ occupied by the blocks $\{ A^{(k)} \}_{k\in P}$ give a block diagonal matrix. 
    \label{def:induced_matching}
\end{definition}

Block diagonal submatrices are closely related to the concept of induced matching in graph theory. In fact, each subset of indices that induces a block diagonal submatrix is an induced matching for the bipartite graph whose adjacency matrix is $G$.

\begin{definition}
    Let $X$ be defined as in~\eqref{eq:block_mat}, and let
    $P_1,\dots,P_T$ be a partition of the block indices $\{1,\dots,q\}$.
    We say that $P_1,\dots,P_T$ is a partition into block diagonal
    submatrices if, for every \newline \mbox{$t=1,\dots,T$}, the blocks $\{ A^{(k)} \}_{k\in P_t}$ form a block diagonal submatrix of $X$ (Definition~\ref{def:induced_matching}).  
    \label{def:block_diag_partition}
\end{definition}

In the example in~\eqref{eq:ex_block}, the partitions into block diagonal
    submatrices are $\{1,3\}\sqcup\{2\}\sqcup\{4\}$, $\{2,3\}\sqcup\{1\}\sqcup\{4\}$, $\{3,4\}\sqcup\{1\}\sqcup\{2\}$, and $\{1\}\sqcup\{2\}\sqcup\{3\}\sqcup\{4\}$. 
    
The following theorem provides an upper bound on the $\rankrel$ of block sparse matrices as in~\eqref{eq:block_mat}.
\begin{theorem} \label{th:sumKrosparse}
    Let $X$, $A^{(k)}$, and $C^{(k)}$ be defined as in~\eqref{eq:block_mat}. Let $G=\sum_{k=1}^q C^{(k)}$ and $P_1,\dots,P_T$ be a partition of the block indices $\{1,\dots,q\}$ into block diagonal
    submatrices of $X$. Then
    \begin{equation}
        \rankrel(X) \leq 1+\rankrel(G) + \sum_{t=1}^T \max_{k \in P_t} \rankrel(A^{(k)}).
        \label{eq:bound_block}
    \end{equation}
\end{theorem}
\begin{proof}
    Similarly to the proof of Theorem~\ref{th:kron}, we construct block factors in the decomposition. We now show how the blocks forming a block diagonal submatrix of $X$ can be exploited to reduce the rank of an RMD of $X$.

    Let $A^{(k)}=\max(0,W^{(k)} H^{(k)})$ be the rank-$r_k$ RMD of $A^{(k)}$. For each $P_t$, we zero pad the columns of $W^{(k)}$ and the rows of $H^{(k)}$ to match $R_t=\max_{k \in P_t} r_k$. We do not change the notation of the zero-padded matrices. For each $t=1,\dots,T$, we allow the blocks in $P_t$ to share the same block columns and block rows in the decomposition; thus, we define 
    \begin{equation*}
     L^{(t)}=\sum_{k \in P_t} e_{i_k} \otimes W^{(k)} \in \mathbb{R}^{pm \times R_t}, \quad \text{and} \quad R^{(t)}=\sum_{\ell \in P_t} e_{j_\ell}^\top \otimes H^{(\ell)} \in \mathbb{R}^{R_t \times sn}. \\
    \end{equation*}
 Computing the product, we obtain
 \begin{equation}
      L^{(t)}R^{(t)}=\sum_{k \in P_t} e_{i_k}e_{j_k}^\top \otimes W^{(k)}H^{(k)} +\sum_{k \in P_t, \ k\neq \ell} e_{i_k}e_{j_\ell}^\top \otimes W^{(k)} H^{(\ell)}. 
      \label{eq:sum_block_LR}
 \end{equation}
 Summing over $t=1,\dots, T$, we get $F=L^{(1)}R^{(1)}+,\dots,+L^{(T)}R^{(T)}$, whose rank is upper bounded by \newline $\sum_{t=1}^T \max_{k \in P_t} \rankrel(A^{(k)})$. 
 
Let $G=\max(0,UV)$ be an RMD of $G$ of rank $\rankrel(G)$.
Let
 \begin{equation}
     \Theta=F+\Phi=F + \delta  (UV-\mathbf{1}_p \mathbf{1}_s^\top) \otimes \mathbf{1}_m \mathbf{1}_n^\top,
     \label{eq:candidate_theta}
 \end{equation}
 be our candidate RMD.
 We now show that $X=\max(0,\Theta)$.
 In what follows, we use the notation $\Theta(i,j)$ to denote the $(i,j)$-th block of $\Theta$, according to the indexing of the entries of $G$.
 \begin{itemize}
     \item If $G_{i,j}=1$, then $\Phi(i,j)=0$. By the definition of $X$, there exists a unique block $A^{(k)}$ in position $(i,j)=(i_k,j_k)$ and since $P_1,\dots,P_T$ is a partition into block diagonal submatrices, the second term in the summation in~\eqref{eq:sum_block_LR} is zero. Therefore, the $(i,j)$-th block of $\Theta$ is $W^{(k)}H^{(k)}$, that is, $\Theta(i_k,j_k)=W^{(k)}H^{(k)}$. 
    \item If $G_{i,j}=0$, then $\Phi(i,j)$ is a componentwise negative matrix with entries not larger than $-\delta$.
    By choosing $\delta$ sufficiently large, we can ensure that $\Theta(i,j)$ is nonpositive.
 \end{itemize}
 Therefore, $\Theta$ is an RMD of $X$, and
 \begin{equation*}
    \rankrel(X) \leq \rank(\Theta) \leq \rank(F) + \rank(\Phi) \leq 1+\rankrel(G)+\sum_{t=1}^T \max_{k \in P_t} \rankrel(A^{(k)}), 
 \end{equation*}
which concludes the proof. 
\end{proof}

Applying the construction shown in the proof to the example in~\eqref{eq:ex_block}, with the partition in block diagonal submatrices $\{1,3\}\sqcup\{2\}\sqcup\{4\}$, we obtain the following an RMD of $X$:
\[
\Theta=\begin{pmatrix}
    -\delta\mathbf 1_m & \delta\mathbf 1_m u_1 & W^{(1)} & W^{(2)} & 0\\
    -\delta\mathbf 1_m & \delta\mathbf 1_m u_2 & W^{(3)} & 0 & 0\\
    -\delta\mathbf 1_m & \delta\mathbf 1_m u_3 & 0 & 0 & W^{(4)}
\end{pmatrix}
\begin{pmatrix}
    \mathbf 1_n^\top & \mathbf 1_n^\top & \mathbf 1_n^\top\\
    v_1\mathbf 1_n^\top & v_2\mathbf 1_n^\top & v_3\mathbf 1_n^\top\\
    H^{(1)} & H^{(3)} & 0\\
    0 & 0 & H^{(2)}\\
    0 & 0 & H^{(4)}
\end{pmatrix},
\]
where the factors $W^{(1)}$ and $W^{(3)}$ can share the same column in the block factor since they form a block diagonal submatrix of $X$.

\section{Exact RMD as a feasibility problem}
\label{sec:algo}

In this section, we present two formulations of exact RMD as feasibility problems. 

\subsection{First formulation: exact RMD as a QCP}

We propose a formulation of exact RMD as a Nonconvex Quadratically Constrained Program (QCP), dubbed ReLU-QCP, by introducing the splitting $\Theta=WH$, with $W\in \mathbb{R}^{m \times r}$ and $H\in \mathbb{R}^{r \times n}$ and constraining the factors $W$ and $H$ to have values in a fixed interval $[-\Bar{M},\Bar{M}]$, with $\Bar{M}$ a positive constant:
\begin{equation}
\begin{aligned}
\text{find}\quad & W \in [-\Bar{M},\Bar{M}]^{m\times r},\quad H \in [-\Bar{M},\Bar{M}]^{r\times n} \\ \text{subject to}\quad & (WH)_{ij}=X_{ij}, \qquad \forall \ (i,j) \ : \ X_{ij}>0\\ 
& (WH)_{ij}\leq 0,  \qquad \forall \ (i,j) \ : \ X_{ij}=0. 
\end{aligned} \label{eq:identity_relu_qcqp_compact} 
\end{equation}
The formulation~\eqref{eq:identity_relu_qcqp_compact} is a feasibility problem as it has no objective function.  
The main reason for introducing such a formulation is that it can be tackled by a global solver such as Gurobi. Such a solver requires the search space to be bounded, and the addition of the constraints $W \in [-\Bar{M},\Bar{M}]^{n\times r}$ and $ H \in [-\Bar{M},\Bar{M}]^{r\times n}$ serves this purpose. Moreover, constraining the factors to be bounded also solves the ill-posedness of RMD highlighted in Proposition~\ref{prop:ill_posed}. 

Listing~\ref{lst:yalmip} gives a MATLAB implementation of~\eqref{eq:identity_relu_qcqp_compact} using the YALMIP toolbox with Gurobi as the underlying solver.

\begin{lstlisting}[float=htbp, caption={YALMIP formulation of the feasibility problem~\eqref{eq:identity_relu_qcqp_compact}. The constant $\Bar{M}$ is denoted Mbar in the code.}, label={lst:yalmip}]
% Decision variables
W = sdpvar(m, r, 'full');
H = sdpvar(r, n, 'full');
P = W*H;                      % bilinear expression

% Bounds and admissibility constraints
pos = X > 0;
F = [-Mbar <= W(:) <= Mbar, -Mbar <= H(:) <= Mbar, ...
     P(pos) == X(pos), ...    % (WH)_ij = X_ij  if X_ij > 0
     P(~pos) <= 0];           % (WH)_ij <= 0   if X_ij = 0

% Feasibility problem solved with Gurobi
ops = sdpsettings('solver', 'gurobi', 'gurobi.NonConvex', 2);
optimize(F, [], ops);
\end{lstlisting}

\subsection{Second formulation based on Latent-RMD}

In his seminal paper~\cite{saul2022nonlinear}, Saul proposes to address Fro-RMD in \eqref{eq:ReLU_NMD} by solving an alternative formulation, dubbed Latent-RMD, where a new latent variable $Z$ satisfying $\max(0,Z)=X$ is introduced, leading to the problem 
\begin{equation}
\argmin_{Z,\Theta} \lVert Z-\Theta \rVert_F^2 \qquad \text{such that } \quad \max(0,Z)=X \quad \text{and} \quad \rank(\Theta) \leq r.
    \label{eq:lat_NMD}
\end{equation}
Latent-RMD has a smooth objective function which is easier to minimize than that of Fro-RMD, and its solutions typically provide accurate approximations of those of Fro-RMD, see~\cite{gillis2025extrapolated,saul2026subzero}. 
We first observe that, given $X \in \mathbb{R}^{m \times n}$, and assuming there exists a rank-$r$ matrix $\Theta \in \mathbb{R}^{m \times n}$ such that $X=\max(0,\Theta)$, for any low-rank solution $\Theta^*$ of~\eqref{eq:ReLU_NMD}, there exists a trivial, feasible latent matrix $Z^*=\Theta^*$, such that $(Z^*,\Theta^*)$ is a solution of~\eqref{eq:lat_NMD} and vice versa.
Therefore, in the exact case, Fro-RMD is equivalent to Latent-RMD, which reduces to a feasibility problem of the form
\begin{equation}
    \text{find } x \text{ in } \mathcal{A} \cap \mathcal{B} , \quad \text{with } \mathcal{A} =\{ Z \in \mathbb{R}^{m \times n} | \max(0,Z)= X \} \ \text{ and } \mathcal{B} =\{ \Theta \in  \mathbb{R}^{m \times n} | \rank(\Theta) \leq r\}.
    \label{eq:feas_prob}
\end{equation}
 It is easy to show that the $\mathcal{A} $ is a closed, convex set, while the $\mathcal{B} $ is closed but nonconvex. 
Feasibility problems, such as the one in~\eqref{eq:feas_prob}, can be solved by the Alternating projection Method (APM), which, starting from a feasible point in one of the sets, produces a sequence of iterates by alternately projecting onto the two sets.

\subsubsection{Naive-RMD and its convergence}

Due to its simplicity, the majority of the available algorithms for Fro-RMD are designed for solving Latent RMD.
Among those, a simple algorithm is the Naive-RMD proposed in~\cite{saul2022nonlinear} by Saul, whose update consists of an alternating optimization between $Z$ and $\Theta$. To our knowledge, no convergence guarantees have been established for the Naive-RMD algorithm, either in the exact or in the approximate case.  

In this section, we propose a novel interpretation of the Naive-RMD as an alternating projection method (APM) for nonconvex feasibility problems.
By using the available theory on APM we establish local convergence for Naive-RMD when applied to exact RMD. We now present in more detail the Naive-RMD algorithm, whose iterative scheme consists of updating $Z$ while $\Theta$ is fixed by solving the subproblem
\mbox{$\argmin_Z \lVert Z-\Theta \rVert_F^2$} such that $  \max(0,Z)=X 
$, whose optimal solution is given by
\begin{equation}
    Z_{ij}=\begin{cases}
        X_{ij} \quad & \mbox{if } X_{ij} > 0\\
        \min(0,\Theta_{ij}) & \mbox{if } X_{ij} = 0.\\        
    \end{cases}
    \label{eq:opt_Z}
\end{equation}
Then, fixing $Z$, it updates $\Theta$ by solving $\argmin_{\Theta} \lVert Z-\Theta \rVert_F^2$ such that $\rank(\Theta) \leq r$, whose optimal solution(s) are
\begin{equation}
    \Theta=\text{TSVD}(Z,r),
    \label{eq:TSVD}
\end{equation}
where TSVD($\cdot$,$r$) denotes a best rank-$r$ approximation of $Z$, obtained by truncating its SVD at the $r$-th singular values.
It is easy to verify that the updates in~\eqref{eq:opt_Z} and in~\eqref{eq:TSVD} are the projections onto $\mathcal{A}$ and $\mathcal{B}$, respectively.
Therefore, the Naive-RMD algorithm can be interpreted as an APM for exact RMD.
 
The APM sequence exhibits global convergence when $\mathcal{A}$ and $\mathcal{B}$ are convex~\cite{bauschke1996projection}, and local convergence in the nonconvex case under additional assumptions on the sets and their intersection. Depending on the properties of the two sets, several convergence results have been established; see~\cite{drusvyatskiy2015transversality,lewis2008alternating,lewis2009local}.
Due to the lack of convexity of $\mathcal{B}$, the Naive algorithm is not guaranteed to be globally convergent. However, local convergence can be directly derived from the contribution by Bellavia et al.~\cite{bellavia2026inexact}, which proposes a specialized analysis for general matrix completion problems.
The feasibility problem that they consider is different from ours since the set $\mathcal{A}$ is substituted by $\{ Z \in \mathbb{R}^{m \times n} \ | \ Z_{ij}=X_{ij} \ \text{for every } \ (i,j) \in \Omega \}$, where $\Omega$ is a known set of indices while $\mathcal{B}$ does not change. Our set $\mathcal{A}$ differs from theirs in that it also involves additional inequality constraints. Nevertheless, both sets share similar properties; in particular, they are both semianalytic, i.e., locally described by finitely many equality and inequality constraints. Therefore,  their analysis can be extended to our setting. We now state the local convergence result, and we refer to~\cite{bellavia2026inexact} for the details.

\begin{theorem}
    Let $\mathcal{A} $ and $\mathcal{B} $ be defined as in~\eqref{eq:feas_prob}, and assume there exists a matrix $Y^*$ of rank exactly $r$ in $\mathcal{A}  \cap \mathcal{B} $. Then there exists a  neighborhood $\mathcal{U}$ of $Y^*$, such that, starting from $\Theta^0 \in \mathcal{U}$, the sequences  $\{Z^k\}_{k \in \mathbb{N}} $ and $\{\Theta^k\}_{k \in \mathbb{N}} $ generated by the Naive-RMD algorithm converge to a matrix $Y \in \mathcal{A}  \cap \mathcal{B} $ with rates $\lVert Y-Z^k \rVert=\mathcal{O}(k^{-\frac{2-\omega}{2 \omega}})$  $\lVert Y-\Theta^k \rVert=\mathcal{O}(k^{-\frac{2-\omega}{2 \omega}})$ for some $\omega \in (0,2)$.
    \label{th:con_Naive_NMD}
\end{theorem} 

The assumption on the existence of $Y^*$ with $\rank$ exactly equal to $r$ ensures that the projection onto $\mathcal{B}$ is single-valued, since $\mathcal{B}$ coincides with the manifold of the matrices with rank exactly $r$. Moreover, the limit point $Y$ does not necessarily coincide with $Y^*$. We also mention that global convergence can be achieved by considering a regularized version of the APM scheme~\cite{bellavia2026inexact}. We did not investigate this point further since it is out of the scope of this work. 
\subsection{Numerical experiments}
\label{sec:exp}
In this subsection, we compare Naive-RMD, applied to~\eqref{eq:feas_prob}, with Gurobi (version 13), called through the YALMIP toolbox in MATLAB and applied to~\eqref{eq:identity_relu_qcqp_compact}. Our first goal is to compare the two proposed approaches in computing exact RMDs. 
Our second goal is to numerically illustrate Theorem~\ref{th:relu_id}, by showing that a rank-3 RMD of the identity matrix can be computed numerically. This problem is known to be difficult to solve using algorithms converging to a stationary point of the inexact formulations such as Fro-RMD or Latent-RMD, due to the large number of spurious minimizers (that is, local but non-global minimizers). 
We show that Gurobi, applied to~\eqref{eq:identity_relu_qcqp_compact}, finds an exact decomposition consistent with the construction used in the proof of Theorem~\ref{th:relu_id}.

Let $(\Tilde{W},\Tilde{H})$ be the solution computed by either Gurobi or Naive-RMD.
We evaluate the performance of the algorithms in terms of the relative Frobenius error defined as
\begin{equation*}
    \err=\frac{\lVert X - \max(0,\Tilde{W}\Tilde{H}) \rVert_F}{\lVert X \rVert_F}.
\end{equation*}
In all the experiments, we stop Naive-RMD if $\err \leq \epsilon$, where $\epsilon$ is a fixed tolerance, or if a maximum number of $10^6$ iterations is reached.
However, this stopping criterion cannot be used for Gurobi, which relies on constraint violation: a solution $(\Tilde{W},\Tilde{H})$ is considered feasible by the solver if the quantity $\lvert X_{i,j}- \max(0,(\Tilde{W}\Tilde{H})_{i,j}) \rvert \leq \delta$ for every $(i,j)$, with $\delta$ a fixed tolerance.
For a fair comparison, we relate the two tolerances $\epsilon$ and $\delta$ by observing that
\begin{equation*}
   \err = \frac{1}{\lVert X\rVert_F}\sqrt{\sum_{i,j}( X_{i,j}- \max(0,(\Tilde{W}\Tilde{H})_{i,j}) )^2} \leq \frac{\sqrt{mn}\delta}{\lVert X\rVert_F}.
\end{equation*}
Thus, considering that Gurobi allows for a tolerance value in $[10^{-9},10^{-2}]$, we choose \newline \mbox{$\delta=\min\left(\max\left(\frac{\epsilon \lVert X\rVert_F}{\sqrt{mn}},10^{-9}\right),10^{-2}\right)$}. Therefore, we ensure that both algorithms are stopped when they reach at least an accuracy of $\epsilon$ on $\err$. The tolerance $\epsilon$ is set to $\epsilon=10^{-9}$. 

All the algorithms are implemented in MATLAB R2021b on a 64-bit Samsung/Galaxy with 11th Gen Intel(R) Core(TM) i5-1135G7 @ 2.40GHz and 8 GB of RAM, under Windows 11 version 23H2.

\subsubsection{Random low-rank matrices}

We consider randomly generated matrices having a fixed $\rankrel$ equal to $r$. In particular, we sample the ground truth factors $W^* \in \mathbb{R}^{m \times r}$ and $H^* \in \mathbb{R}^{r \times n}$ from a uniform distribution in $[-0.5,0.5]$, that is, \texttt{W$^*$ = rand(m,r)-0.5} and \texttt{H$^*$ = rand(r,n)-0.5} in MATLAB, and we compute $X^*=\max(0,W^* H^*)$. We set the bound on the entries of the factor $W$ and $H$ for the feasibility problem formulation in~\eqref{eq:identity_relu_qcqp_compact} to $\Bar{M}=0.5$, since we know that at least one solution in that range exists. We choose different dimensions $m$ and $n$, while choosing $r \in \{5,10,15\}$. Then, we assume $W^*$ and $H^*$ are unknown, and we compare Naive-RMD and Gurobi in recovering an exact RMD of $X^*$. The Naive-RMD algorithm is randomly initialized by taking $\Theta^0=W^0 H^0$, where $W^0$ and $H^0$ are generated following the same strategy used to generate $W^*$ and $H^*$. 
We also set a time limit of 300 seconds, after which the execution of the methods is stopped. We show the average time needed to compute each decomposition for different problem dimensions in Table \ref{tab:Naive_Gurobi_comp}. If the solver does not find any solution in the allotted time, we report it as a failure.

\begin{table}[ht]
    \centering
   \begin{tabular}{cc cc cc}
    \toprule
    \multirow{2}{*}{Size} & \multirow{2}{*}{Rank} & \multicolumn{2}{c}{Gurobi} & \multicolumn{2}{c}{Naive-RMD} \\
    \cmidrule(lr){3-4} \cmidrule(lr){5-6}
     & & Time (s) & Fails & Time (s) & Fails \\
      \midrule
    $25 \times 20$ & 5 & 0.22 & 0 & 0.56 & 0 \\
    $30 \times 25$ & 5 & 18.93 & 0 & 0.21 & 0 \\
    $40 \times 35$ & 5 & 52.33 & 1 & 0.20 & 0 \\
    $50 \times 45$ & 5 & 2.78 & 0 & 0.21 & 0 \\
    $60 \times 50$ & 5 & 4.33 & 0 & 0.19 & 0 \\
     $70 \times 60$ & 5 & 8.64 & 0 & 0.24 & 0 \\
    \midrule
    $25 \times 20$ & 10 & 0.34 & 2 & 0.13 & 0 \\
    $30 \times 25$ & 10 & 59.70 & 5 & 3.21 & 1 \\
    $40 \times 35$ & 10 & 8.32 & 0 & 6.68 & 1 \\
    $50 \times 45$ & 10 & 17.37 & 0 & 0.46 & 0 \\
    $60 \times 50$ & 10 & 22.92 & 0 & 0.32 & 0 \\
     $70 \times 60$ & 10 & 33.66 & 3 & 0.30 & 0 \\
     \midrule
    $25 \times 20$ & 15 & 0.47 & 8 & 0.01 & 0 \\
    $30 \times 25$ & 15 & 4.29 & 0 & 0.03 & 0 \\
    $40 \times 35$ & 15 & 15.78 & 0 & 0.40 & 0 \\
    $50 \times 45$ & 15 & 28.00 & 0 & 6.91 & 0 \\
    $60 \times 50$ & 15 & 41.02 & 1 & 30.66 & 0 \\ 
     $70 \times 60$ & 15 & 38.39 & 8 & 1.00 & 0 \\
    \bottomrule
\end{tabular}
    \caption{Average running time performed on the successful runs to compute the exact RMD over 10 randomly generated matrices admitting an exact RMD of rank $r \in \{5, 10, 15\}$. The column "Fails" shows the number of runs for which each algorithm could not find a solution within 300 seconds. The time is in seconds, and it refers to the CPU time used by the Gurobi solver, while the time needed to build the model by the YALMIP model is not considered in this table.}
    \label{tab:Naive_Gurobi_comp}
\end{table}
Table~\ref{tab:Naive_Gurobi_comp} shows that Naive-RMD is significantly faster than the Gurobi solver in 14 out of the 15 analyzed cases. Moreover, the Gurobi solver fails 8 times when $r=15$ and the dimensions are $25 \times 20$ and $70 \times 60$. On the other hand, Naive-RMD fails just two times in al the experiments considered. We also observe that increasing the dimensions of the matrix does not automatically imply that the problem becomes harder. In fact, finding a rank-10 RMD of a $50 \times 45$ or a $60 \times 50$ matrix seems to be easier than the case with dimensions $25 \times 20$ and $30 \times 25$ in the setting we consider. 
These results for Naive-RMD seem to indicate that randomly generated instances have a benign landscape in many cases. 
The computational experiments in~\cite{gillis2025extrapolated} also support this conjecture, as all synthetically generated matrices could be factorized exactly with iterative algorithms.   
This is a topic for further research.

\subsubsection{Identity matrix} We now aim to compute an exact RMD of the identity matrix $I_n$ of size $n \times n$. 
We know from Theorem~\ref{th:relu_id} that, independently of $n$, there exists a rank-3 RMD of the identity matrix. Nevertheless, computing the decomposition using numerical methods is not trivial, and many of the methods in the literature fail to achieve this task. In what follows, we fix $n=10$, and we solve the problem in~\eqref{eq:identity_relu_qcqp_compact} with $X=I_n$. First, we run Naive-RMD with 1000 different random initializations $\Theta^0=W^0 H^0$, where $W^0$ and $H^0$ are sampled from a Gaussian distribution.
As in the previous experiment, we set a maximum iteration number of $10^5$. The average relative error achieved by the Naive-RMD solution is $1.76 \cdot 10^{-1}$ and none of the initializations yields an error below $10^{-3}$. The Naive-RMD algorithm fails to find an exact $\rank$-3 solution of the 10-by-10 identity matrix, when randomly initialized. 

 We now test the Gurobi solver applied to~\eqref{eq:identity_relu_qcqp_compact}, where the bound on the entries of the factors $W$ and $H$ is set to $\Bar{M}=10$. The solver finds a solution $(\Hat{W},\Hat{H})$ with a relative error of the order of $10^{-13}$. Following the points-hyperplanes interpretation given in Section~\ref{sec:exact_NMD}, Figure~\ref{fig:gurobi_id} shows, in a three-dimensional space, the points corresponding to the rows of $\Hat{W}$ and two hyperplanes whose normal vectors correspond to two columns of $\Hat{H}$. 
\begin{figure}[h]
    \centering
    \includegraphics[width=0.7\linewidth]{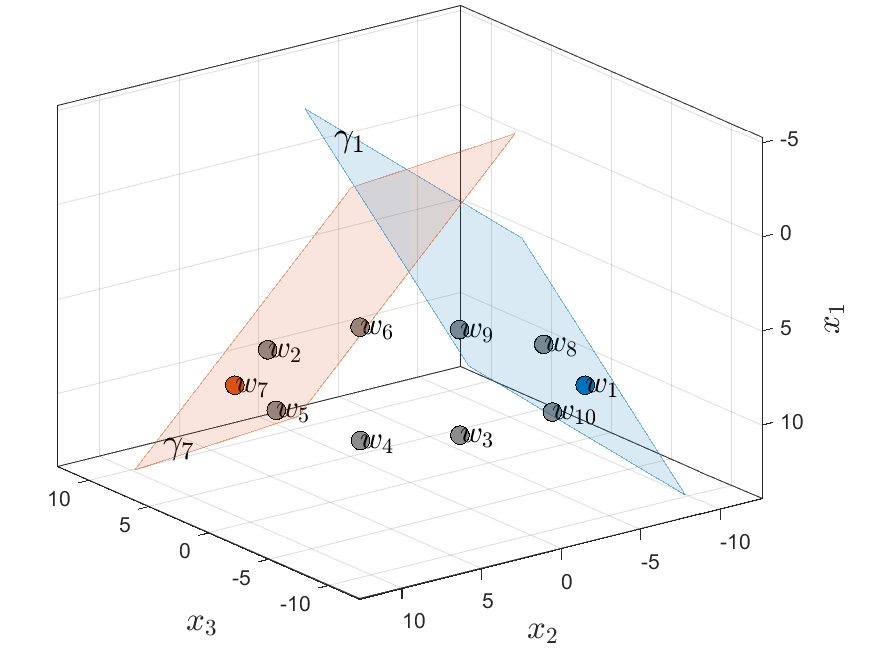}
    \caption{Illustration of the rank-3 RMD $\Hat{W}\Hat{H}$ of $I_{10}$ obtained using Gurobi. We display the rows of $\Hat{W}$ as the points $w_i$'s, and the hyperplanes $\gamma_1$ and $\gamma_7$ whose normal vectors are the first and seventh columns of $\Hat{H}$, resp.; see the geometric interpretation of exact RMD in Section~\ref{sec:exact_NMD}. In fact, each hyperplane $\gamma_i$ for $i=1,\dots,10$ separates the points $w_i$ from all the others, $w_j$ with $j \neq i$, in order to recover the sign pattern of the corresponding $i$-th column of $I_{10}$.} 
    \label{fig:gurobi_id}
\end{figure}
Interestingly, the rows of $\Hat{W}$ in the Gurobi solution represent (approximately) coplanar points lying on a circle, which reflects the geometric construction used in~\cite{delsarte1989low} and in~\cite{saul2022nonlinear} to prove that $\ranksg(I_n)$ and $\rankrel(I_n)$ are smaller than 3, respectively. According to the geometric interpretation of RMD, the columns of $\Hat{H}$ are normal vectors to hyperplanes that separate each point from all the others; thus, they have a positive scalar product with only one point  (diagonal element) and a negative one with all the others (off diagonal elements). 

Contrary to the exact Gurobi solver, the solution computed by Naive-RMD depends on the starting point of the algorithm, see Theorem~\ref{th:con_Naive_NMD}. We verify the local behavior of Naive-RMD by considering, as a starting point, the exact Gurobi solution $\Hat{W}\Hat{H}$ where we increasingly flip the signs of the entries of the matrix. According to the geometric representation of RMD, flipping the sign in the exact solution means moving the points-hyperplanes representation by randomly switching the scalar product between points and the normal vector to hyperplanes separating them. Therefore, the more signs are flipped, the further the starting point is from the solution. In Figure~\ref{fig:id_Naive_noise}, we evaluate the final relative error for an increasing  number of entries with a flipped sign. The results are averaged over 10 starting points, where the entries with flipped signs are randomly selected. 
\begin{figure}[h]
    \centering
    \includegraphics[width=0.5\linewidth]{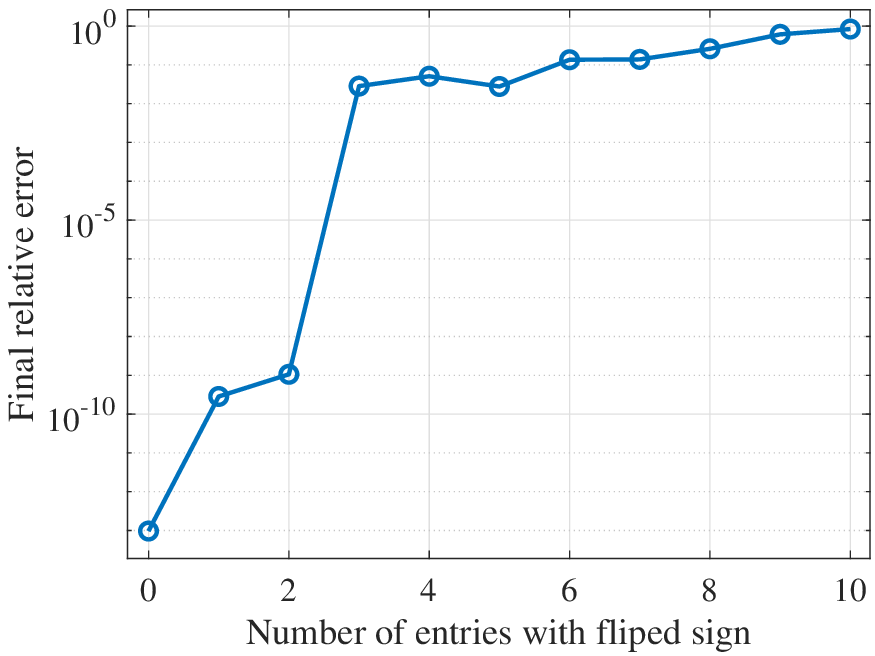}
    \caption{Average final relative error of the solution computed using Naive-RMD with starting point obtained from the exact Gurobi solution  $\Hat{\Theta}=\Hat{W}\Hat{H}$ in which the sign of an increasing number entries are flipped. The entries in $\Hat{\Theta}$ presenting a flipped sign are randomly selected at each run.} 
    \label{fig:id_Naive_noise}
\end{figure}
As expected, if the initial point is not too far from an optimal solution, that is, if fewer than three entries are flipped, Naive-RMD is still able to find an accurate solution.
In contrast, with three or more flipped entries, Naive-RMD does not converge to an exact solution, and its final relative error exceeds $10^{-2}$.

\section{Conclusion}

In this paper, we addressed several open questions in the context of RMD. First, we proved the NP-hardness and ill-posedness of Fro-RMD, which is a standard formulation used to compute approximate RMDs. Then, we considered exact RMD, introducing the notion of $\rankrel$, and providing a geometric interpretation as a points-hyperplanes system. We used this geometric intuition to give upper bounds on the $\rankrel$ of special classes of matrices. Finally, we proposed two new formulations of exact RMD as a feasibility problem and showed that the Naive-RMD algorithm by Saul can be employed to solve one of the two. We established local convergence guarantees for the Naive-RMD algorithm when applied to exact RMD.

\paragraph{Use of AI assistance} During the preparation of this manuscript, the authors used Claude (Anthropic) to improve the presentation of several proofs and to check the clarity and logical consistency of mathematical arguments. The AI system was used solely as a writing and reasoning aid, and the key ideas behind this work are not AI generated. All mathematical results, proofs, and conclusions were independently verified by the authors, who take full responsibility for the content of this manuscript.

\bibliographystyle{ieeetr}
\bibliography{bib.bib}

@article{delsarte1989low,
  title={Low rank matrices with a given sign pattern},
  author={Delsarte, Philippe and Kamp, Yves},
  journal={SIAM Journal on Discrete Mathematics},
  volume={2},
  number={1},
  pages={51--63},
  year={1989},
  publisher={SIAM}
}

@article{saul2022nonlinear,
  title={A nonlinear matrix decomposition for mining the zeros of sparse data},
  author={Saul, Lawrence K},
  journal={SIAM Journal on Mathematics of Data Science},
  volume={4},
  number={2},
  pages={431--463},
  year={2022},
  publisher={SIAM}
}

@inproceedings{Thayse2026,
  title={Sigmoid Nonlinear Matrix Decomposition with Cross Entropy},
  author={Thayse, Alex and Vandaele, Arnaud and Gillis, Nicolas},
  booktitle={European Signal Processing Conference},
  year={2026} 
}

@inproceedings{lefebvre2024component,
  title={Component-wise squared factorization},
  author={Lefebvre, Joakim and Vandaele, Arnaud and Gillis, Nicolas},
  booktitle={34th International Workshop on Machine Learning for Signal Processing (MLSP)},
  pages={1--6},
  year={2024},
  organization={IEEE}
}

@article{saul2022geometrical,
  title={A geometrical connection between sparse and low-rank matrices and its application to manifold learning},
  author={Saul, Lawrence K},
  journal={Transactions on Machine Learning Research},
  year={2022}
}

@inproceedings{seraghiti2023accelerated,
  title={{Accelerated Algorithms For Nonlinear Matrix Decomposition With The ReLU Function}},
  author={Seraghiti, Giovanni and Awari, Atharva and Vandaele, Arnaud and Porcelli, Margherita and Gillis, Nicolas},
  booktitle={33rd International Workshop on Machine Learning for Signal Processing (MLSP)},
  pages={1--6},
  year={2023},
  organization={IEEE}
}

@inproceedings{naik2022truncated,
  title={{Truncated Matrix Completion-An Empirical Study}},
  author={Naik, Rishhabh and Trivedi, Nisarg and Tarzanagh, Davoud Ataee and Balzano, Laura},
  booktitle={30th European Signal Processing Conference (EUSIPCO)},
  pages={847--851},
  year={2022},
  organization={IEEE}
}

@inproceedings{mazumdar2019learning,
  title={Learning and recovery in the {ReLU} model},
  author={Mazumdar, Arya and Rawat, Ankit Singh},
  booktitle={57th Annual Allerton Conference on Communication, Control, and Computing (Allerton)},
  pages={108--115},
  year={2019},
  organization={IEEE}
}

@inproceedings{goel2017reliably,
  title={Reliably learning the {ReLU} in polynomial time},
  author={Goel, Surbhi and Kanade, Varun and Klivans, Adam and Thaler, Justin},
  booktitle={Conference on Learning Theory},
  pages={1004--1042},
  year={2017},
  organization={PMLR}
}

@article{wang2024momentum,
  title={{A Momentum Accelerated Algorithm for ReLU-based Nonlinear Matrix Decomposition}},
   author={Wang, Qingsong and Cui, Chunfeng and Han, Deren},
  journal={IEEE Signal Processing Letters},
  volume={31},
  pages={2865--2869},
  year={2024},
  publisher={IEEE}
}

@article{liu2024symmetric,
  title={{Symmetric Matrix Completion with ReLU Sampling}},
  author={Liu, Huikang and Wang, Peng and Huang, Longxiu and Qu, Qing and Balzano, Laura},
  journal={arXiv preprint arXiv:2406.05822},
  year={2024}
}

@article{ciaperoni2024hadamard,
  title={The {H}adamard decomposition problem},
  author={Ciaperoni, Martino and Gionis, Aristides and Mannila, Heikki},
  journal={Data Mining and Knowledge Discovery},
  pages={1--42},
  year={2024},
  publisher={Springer}
}

@article{gillis2018complexity,
  title={On the complexity of robust {PCA} and $\ell_1$-norm low-rank matrix approximation},
  author={Gillis, Nicolas and Vavasis, Stephen A},
  journal={Mathematics of Operations Research},
  volume={43},
  number={4},
  pages={1072--1084},
  year={2018},
  publisher={INFORMS}
}

@inproceedings{hatami2022lower,
  title={Lower bound methods for sign-rank and their limitations},
  author={Hatami, Hamed and Hatami, Pooya and Pires, William and Tao, Ran and Zhao, Rosie},
  booktitle={Approximation, Randomization, and Combinatorial Optimization. Algorithms and Techniques (APPROX/RANDOM)},
  pages={22--1},
  year={2022},
  organization={Schloss Dagstuhl--Leibniz-Zentrum f{\"u}r Informatik}
}

@article{gillis2025extrapolated,
  title={An extrapolated and provably convergent algorithm for nonlinear matrix decomposition with the {ReLU} function},
  author={Gillis, Nicolas and Porcelli, Margherita and Seraghiti, Giovanni},
  journal={SIAM Journal on Mathematics of Data Science},
  volume={8},
  number={3},
  pages={738--765},
  year={2026},
  publisher={SIAM}
}

@article{gillis2026manifold,
  title={Manifold-based {A}lgorithms for the {H}adamard {D}ecomposition},
  author={Gillis, Nicolas and Saha, Subhayan and Sicilia, Stefano and Vandaele, Arnaud},
  journal={arXiv preprint arXiv:2605.28980},
  year={2026}
}

@article{awari2025alternating,
  title={{A}lternating {D}irection {M}ethod of {M}ultipliers for {N}onlinear {M}atrix {D}ecompositions},
  author={Awari, Atharva and Gillis, Nicolas and Vandaele, Arnaud},
  journal={arXiv preprint arXiv:2512.17473},
  year={2025}
}

@inproceedings{alon2016sign,
  title={Sign rank versus {VC} dimension},
  author={Alon, Noga and Moran, Shay and Yehudayoff, Amir},
  booktitle={Conference on Learning Theory},
  pages={47--80},
  year={2016},
  organization={PMLR}
}

@article{paturi1986probabilistic,
  title={Probabilistic communication complexity},
  author={Paturi, Ramamohan and Simon, Janos},
  journal={Journal of Computer and System Sciences},
  volume={33},
  number={1},
  pages={106--123},
  year={1986},
  publisher={Elsevier}
}

@article{bauschke1996projection,
  title={On projection algorithms for solving convex feasibility problems},
  author={Bauschke, Heinz H and Borwein, Jonathan M},
  journal={SIAM review},
  volume={38},
  number={3},
  pages={367--426},
  year={1996},
  publisher={SIAM}
}

@article{bellavia2026inexact,
  title={An inexact alternating projection method with application to matrix completion},
  author={Bellavia, Stefania and Rebegoldi, Simone and Silei, Mattia},
  journal={Computational Optimization and Applications},
  pages={1--46},
  year={2026},
  publisher={Springer}
}

@inproceedings{awari2024coordinate,
  title={Coordinate descent algorithm for nonlinear matrix decomposition with the {ReLU} function},
  author={Awari, Atharva and Nguyen, Harrison and Wertz, Samuel and Vandaele, Arnaud and Gillis, Nicolas},
  booktitle={32nd European Signal Processing Conference (EUSIPCO)},
  pages={2622--2626},
  year={2024},
  organization={IEEE}
}

@article{wang2025accelerated,
  title={An accelerated alternating partial {B}regman algorithm for {ReLU}-based matrix decomposition},
  author={Wang, Qingsong and Qu, Yunfei and Cui, Chunfeng and Han, Deren},
  journal={Journal of Scientific Computing},
  volume={105},
  number={2},
  pages={48},
  year={2025},
  publisher={Springer}
}

@inproceedings{alon1985geometrical,
  title={Geometrical realization of set systems and probabilistic communication complexity},
  author={Alon, Noga and Frankl, Peter and Rodl, Vojtech},
  booktitle={26th Annual Symposium on Foundations of Computer Science},
  pages={277--280},
  year={1985},
  organization={IEEE}
}

@article{ben2002limitations,
  title={Limitations of learning via embeddings in {E}uclidean half spaces},
  author={Ben-David, Shai and Eiron, Nadav and Simon, Hans Ulrich},
  journal={Journal of Machine Learning Research},
  volume={3},
  number={Nov},
  pages={441--461},
  year={2002}
}

@article{klivans2007unconditional,
  title={Unconditional lower bounds for learning intersections of halfspaces},
  author={Klivans, Adam R and Sherstov, Alexander A},
  journal={Machine Learning},
  volume={69},
  number={2},
  pages={97--114},
  year={2007},
  publisher={Springer}
}

@inproceedings{feldman2017general,
  title={A general characterization of the statistical query complexity},
  author={Feldman, Vitaly},
  booktitle={Conference on learning theory},
  pages={785--830},
  year={2017},
  organization={PMLR}
}

@article{chattopadhyay2018separation,
  title={Separation of unbounded-error models in multi-party communication complexity},
  author={Chattopadhyay, Arkadev and Mande, Nikhil S},
  journal={Theory of Computing},
  volume={14},
  number={1},
  pages={1--23},
  year={2018},
  publisher={Theory of Computing Exchange}
}

@article{sherstov2011unbounded,
  title={The unbounded-error communication complexity of symmetric functions},
  author={Sherstov, Alexander A},
  journal={Combinatorica},
  volume={31},
  number={5},
  pages={583--614},
  year={2011},
  publisher={Springer}
}

@article{fox2010overlap,
  title={Overlap properties of geometric expanders},
  author={Fox, Jacob and Gromov, Mikhail and Lafforgue, Vincent and Naor, Assaf and Pach, J{\'a}nos},
  journal={arXiv preprint arXiv:1005.1392},
  year={2010}
}

@article{saul2026subzero,
  title={Subzero matrix completion for sparse data analysis: large-scale learning of latent low-rank structure},
  author={Saul, Lawrence K and Huang, Ningyuan and Bollweg, Dennis and Soules, Jeff and Halikias, Diana C},
  journal={arXiv preprint arXiv:2608.21607},
  year={2026}
}

@article{drusvyatskiy2015transversality,
  title={Transversality and alternating projections for nonconvex sets},
  author={Drusvyatskiy, Dmitriy and Ioffe, Alexander D and Lewis, Adrian S},
  journal={Foundations of Computational Mathematics},
  volume={15},
  number={6},
  pages={1637--1651},
  year={2015},
  publisher={Springer}
}

@article{lewis2008alternating,
  title={Alternating projections on manifolds},
  author={Lewis, Adrian S and Malick, J{\'e}r{\^o}me},
  journal={Mathematics of Operations Research},
  volume={33},
  number={1},
  pages={216--234},
  year={2008},
  publisher={INFORMS}
}

@article{lewis2009local,
  title={Local linear convergence for alternating and averaged nonconvex projections},
  author={Lewis, Adrian S and Luke, D Russell and Malick, J{\'e}r{\^o}me},
  journal={Foundations of Computational Mathematics},
  volume={9},
  number={4},
  pages={485--513},
  year={2009},
  publisher={Springer}
}

@article{gillis2026complexity,
  title={{On the Complexity of Entrywise Power Matrix Factorization}},
  author={Gillis, Nicolas and Saha, Subhayan and Sicilia, Stefano and Vandaele, Arnaud},
  journal={arXiv preprint arXiv:2607.04875},
  year={2026}
}

@inproceedings{Nguyen,
  author={Nguyen, Harrison and Awari, Atharva and Vandaele, Arnaud and Gillis, Nicolas},
  booktitle={35th International Workshop on Machine Learning for Signal Processing (MLSP)}, 
  title={Nonlinear Matrix Decomposition with the Sigmoid Function}, 
  year={2025},
  volume={},
  number={},
  pages={1-6},
  doi={10.1109/MLSP62443.2025.11204338}}

@article{droge2023kissing,
  title={Kissing to find a match: Efficient low-rank permutation representation},
  author={Dr{\"o}ge, Hannah and L{\"a}hner, Zorah and Bahat, Yuval and Martorell Nadal, Onofre and Heide, Felix and M{\"o}ller, Michael},
  journal={Advances in Neural Information Processing Systems},
  volume={36},
  pages={48459--48471},
  year={2023}
}
\end{document}